\documentclass[12pt]{amsart}
\usepackage[pagewise]{lineno}%\linenumbers
\usepackage{lipsum}

\newcommand\blfootnote[1]{%
  \begingroup
  \renewcommand\thefootnote{}\footnote{#1}%
  \addtocounter{footnote}{-1}%
  \endgroup
}
\usepackage{cases}
\usepackage{float}
\usepackage{hyperref}
\usepackage{blindtext}
\usepackage{multicol}
\usepackage{graphics}
\usepackage{graphicx}
\usepackage{epstopdf}
\usepackage{amssymb}
\usepackage{amsmath}
\usepackage{amsthm}
\usepackage{array}
\usepackage{latexsym}
\usepackage{amsfonts}
\usepackage{color}
\usepackage{verbatim}
\usepackage[numbers]{natbib}
\usepackage{amsthm,url,xspace}
\definecolor{couleurCitations}{rgb}{0,0,0.85}
\definecolor{couleurRef}{rgb}{0.75,0,0}

\newtheorem{lema}{Lemma}
\newtheorem{teo}{Theorem}

\title{On non almost-fibered knots}
\author[Mario Eudave-Mu\~noz, Araceli Guzm\'an-Trist\'an, Enrique Ram\'irez-Losada]{Mario Eudave-Mu\~noz, Araceli Guzm\'an-Trist\'an, Enrique Ram\'irez-Losada}
\address{Instituto de Matem\'aticas, Universidad Nacional Aut\'onoma de M\'exico, Campus Juriquilla, Quer\'etaro, M\'exico}
\email{mario@matem.unam.mx}

\address{Centro de Investigaci\'on en Matem\'aticas, A.C. 402, 36000 Guanajuato, Gto., M\'exico}
\email{araceli.guzman@cimat.mx}

\address{Centro de Investigaci\'on en Matem\'aticas, A.C. 402, 36000 Guanajuato, Gto., M\'exico}
\email{kikis@cimat.mx}

\begin{document}  

    
\begin{abstract}
 An almost-fibered knot is a knot whose complement possesses a circular thin position in which there is one and only one weakly incompressible Seifert surface and one incompressible Seifert surface. Infinite examples of almost-fibered knots are known. In this article, we show the existence of infinitely many hyperbolic genus one knots that are not almost-fibered. 
\end{abstract}

\maketitle

\section{Introduction}

Let $K$ be a knot in $S^3$. A regular circle-valued Morse function on the knot complement $f:S^3\backslash K\to S^1$ induces a handle decomposition on the knot exterior. In \cite{Fabiola}, F. Manjarrez-Guti\'errez adapted the ideas of \cite{ST} to re-order the handles in such a way that the regular level surfaces are as simple as possible, giving rise to the definition of \emph{circular width} and \emph{circular thin position of the knot exterior}. Circular thin position of the knot exterior gives a sequence of Seifert surfaces which are alternately incompressible and weakly incompressible. In this context a \emph{fibered knot} is a knot whose exterior has a circular thin position with one and only one incompressible level surface and none weakly incompressible level surface. This is the unique circular thin position for a fibered knot (see \cite{Burde} or \cite{Whitten}). The circular width of a fibered knot is defined to be zero. An almost-fibered knot is a knot whose complement possesses a circular thin position in which there is one and only one weakly incompressible Seifert surface and one incompressible Seifert surface. Goda \cite{Goda} showed that all prime non-fibered knots up to ten crossings are almost-fibered knots. On the other hand, F. Manjarrez-Guti\'errez, V. N\'uñez and the third author \cite{FabKikPis} showed that all non-fibered knots which are free genus one, are almost-fibered knots.  It is worth mentioning that any non-fibered knot has a circular decomposition with only one thick level (see Subsection \ref{thinthick} for a definition), that is, any knot has a decomposition that looks like an almost-fibered knot, except that the thick surface may not be weakly incompresible, then the decomposition is not thin.

Until now, the only known examples of non almost-fibered knots are some connected sum of knots, for instance, the connected sum of $a$-small knots \cite{Fabiola2}. In 2017, at the Mathematical Congress of Americas, Hans Boden asked whether there are non almost-fibered, hyperbolic knots. In this paper we answer this question in the affirmative.

\begin{teo}
There exists an infinite family of genus one hyperbolic knots which are not almost-fibered. 
\end{teo}

The paper is organized as follows. In Section \ref{prelim} we review definitions about circular thin position of knots as well as definitions concerning to companion annuli in genus two handlebodies. In Section \ref{ejemplos} we give an explicit construction of the desired family of knots, and in Section \ref{width} we develop the proofs of the main results.

\section{Preliminaries}\label{prelim}

\subsection{Knots and surfaces.}

Let $K$ be a knot in $S^3$. The knot complement will be denoted by $C_K=S^3\backslash K$. A closed regular neighborhood of $K$ will be denoted by $\mathcal{N}(K)$ and the exterior of the knot $K$ by $E(K)=S^3\backslash \mathcal{N}(K)^{\mathrm{o}}$.

A Seifert surface for a knot $K$ is an oriented compact surface whose boundary is the knot $K$. It should be noted that such a surface always exists for any knot in $S^3$.

%If $F$ is a two-sided surface, we say that $F$ is \emph{compressible} if there exists a disk $D\subset E(K)$ to either side of $F$ such that $D\cap int (F)=\partial D$ does not bound a disk in $F$, that is $\partial D$ is an essential closed curve in $F$. Such a disk $D$ is called a \emph{compressing disk for} $F$ in $E(K)$. If there are no compressing disks for $F$ in $E(K)$, we say that $F$ is \emph{incompressible in} $E(K)$.

%$F$ is said to be $\partial$-\emph{compressible} if there is a disk $D\subset E(K)$ such that
%\begin{itemize}
%    \item $D\cap F=\alpha\subset\partial D$, $D\cap \partial E(K)=\beta\subset \partial D$
   % \item $\alpha\cup\beta=\partial D$ with $\alpha\cap\beta=\partial\alpha=\partial\beta$ and
    %\item $\alpha$ does not cobound a disk with another arc in $\partial F$; that is, $\alpha$ is essential in $F$.
%\end{itemize}
%Such a disk $D$ above is called a $\partial$-\emph{compressing disk} for $F$ in $E(K)$.

If $F$ is a two-sided surface, we say that $F$ is \emph{weakly incompressible} if any two compressing disks for $F$ on opposite sides of $F$ intersect along their boundary.

\subsection{Circular thin position for knots.}\label{thinthick}

In this section, much of the notation and definitions follow from F. Manjarrez-Gutiérrez \cite{Fabiola}.

Let $K\subset S^3$ be a knot. Let $F\colon S^3\backslash K\to S^1$ be a Morse function and define $f\colon E(K)\to S^1$ to be the restriction of $F$ to $E(K)$. A fundamental cobordism of $f$ can be isotoped to have no local maxima or minima \cite{Milnor}, we assume that all critical points of $f$ have index $1$ or $2$.

We construct a handle decomposition of $E(K)$ from $f$ as follows: Choose a Seifert surface $R$ at a regular level of $f$ between an index-1 and an index-2 critical points. There are many such 
choices: We assume $R$ is chosen to have smallest genus among all choices. If $f$ has no critical points, then $K$ is fibered and there is only one choice for $R$. Otherwise, we see that the 
critical points of $f$ define collections $N=\{N_{1},N_{2},\ldots,N_{k}\}$ and $B=\{B_{1},B_{2},\ldots,B_{k}\}$ of 1- and 2-handles, respectively. We assume the handles in $N_{1}$ appear first 
after $R$ and, moreover, that the handles in $N_{i}$ appear before the handles in $B_{i}$ and that the handles in $B_{i}$ appear before the handles in $N_{i+1}$ (taking all indices modulo $k$ where necessary). Construct the compact manifold

$$
H=(R\times I)\cup (N_{1}\cup B_{1})\cup\ldots\cup (N_{k}\cup B_{k})
$$
\noindent by flowing $R$ along $E(K)$ via the gradient of $f$, attaching handles from $N$ and $B$ as prescribed by the critical points of $f$. Choose a regular level $S_{i}$ separating $N_{i}$ from $B_{i}$; that is,

$$
S_{i}\cong\overline{\partial[(R\times I)\cup(N_{1}\cup B_{1})\cup\ldots\cup N_{i}]\backslash[\partial E(K)\cup (R\times \{0\})]}
$$ 

Call $S_{i}$ a \emph{thick level} of $f$ and set $\mathcal{S}=\bigcup_{i=1}^{k} S_{i}$, similarly choose a regular level $F_{i}$ separating $B_{i}$ from $N_{i+1}$; that is,

$$
F_{i}=\overline{\partial[(R\times I)\cup (N_{1}\cup B_{1})\cup\ldots\cup (N_{i}\cup B_{i})]\backslash [\partial E(K)\cup (R\times \{0\})]}
$$

Call $F_{i}$ a \emph{thin level} of $f$ and set $\mathcal{F}=\bigcup_{i=1}^{k} F_{i}$. 

We also define

$$
W_{i}=(\text{collar of } F_{i})\cup (N_{i}\cup B_{i})
$$

\noindent which is a $3$-manifold with boundary $\partial W_{i}=F_{i}\cup F_{i+1}\cup (W_{i}\cap\partial E(K))$. The thick level of $S_{i}$ defines a compact (but not closed) Heegaard surface for $W_{i}$, dividing it into compression bodies

$$
A_{i}=(\text{collar of } F_{i})\cup N_{i} \hspace{1cm}\text{and}\hspace{1cm} G_{i}=(\text{collar of } S_{i})\cup B_{i}
$$

The boundary $\partial A_{i}$ can be seen as the union of three components; that is, $\partial A_{i}=S_{i}\cup F_{i}\cup\partial_{v} A_{i}$,	where $\partial_{v} A_{i}=\partial A_{i}\cap\partial E(K)$. We call $\partial_{v} A_{i}$ the \emph{vertical boundary} of $A_{i}$. We can similarly define $\partial_{v} G_{i}=\partial G_{i}\cap\partial E(K)$ to be the vertical boundary of $G_{i}$. Note that the vertical boundary is an annulus.

Observe that $F_{k}$ is diffeomorphic to $R$. The function $f$ defines a diffeomorphism $\phi:R\to R$. When $K$ is fibered, $\phi$ is called the \emph{monodromy} of $K$. In this case, we see that 

$$
E(K)=H/(R\times\{0\})=\phi (R\times\{1\})
$$

The collection $\mathcal{D}=\{(W_{i};A_{i},G_{i})\}_{i=1}^{k}$ will be called the \emph{circular handle decomposition} of $E(K)$ induced by $f$.

\begin{figure}
\begin{center}
\includegraphics[scale=.4]{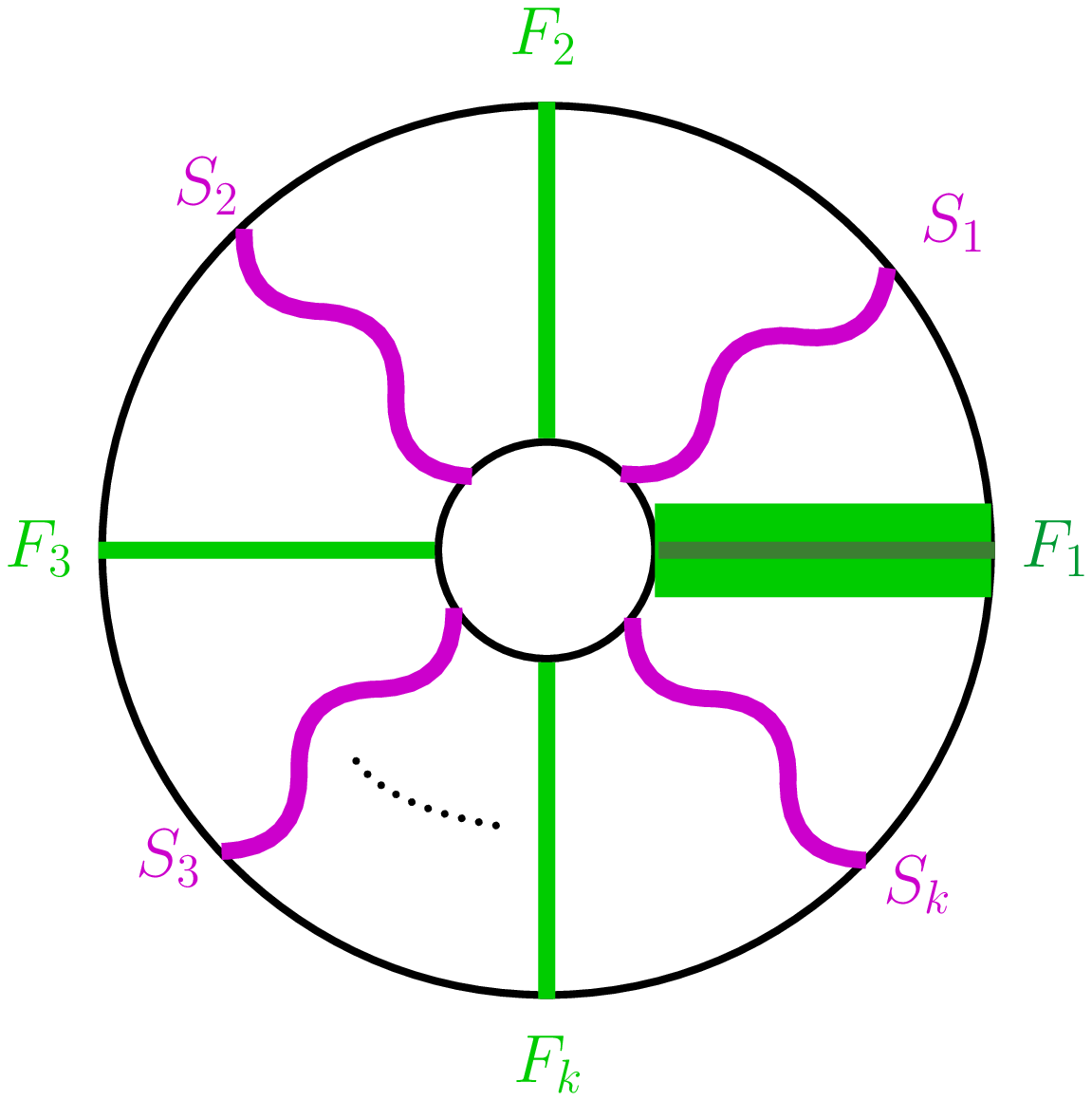} \end{center}
\caption{}
\label{}
\end{figure}

For a closed, connected surface $S\neq S^{2}$ define its \emph{complexity} $c(S)=1-\chi(S)$. If $S$ has nonempty boundary, define $c(S)=1-\chi(\overline{S})$, where $\overline{S}$ denotes $S$ with its  boundary components capped off with disks. We define $c(S^{2})=0$ and $c(D^{2})=0$. If $f$ is disconnected, define $c(S)=\sum c(S_{i})$, where $S=\bigcup_{i} S_{i}$ and each $S_{i}$ is connected.

For a knot $K\subset S^{3}$ with circular handle decomposition $\mathcal{D}$ for $E(K)$, the \emph{circular width} $cw(\mathcal{D})$ of the decomposition $\mathcal{D}$ is the multiset of integers $\{c(S_{i})\}_{i=1}^{k}$, ordered in a non-increasing way, and $|cw(\mathcal{D})|=k$ is the number of thick levels in $\mathcal{D}$. The \emph{circular width } $cw(E(K))$ \emph{of the knot exterior} $E(K)$ is defined to be the minimal circular width among all circular handle decompositions of $E(K)$. The minimum is taken using the lexicographic ordering of multisets of integers.

The pair $(E(K),\mathcal{D})$ is in \emph{circular thin position} if $\mathcal{D}$ realizes the circular width of $E(K)$. When $K$ is a fibered knot, we define $cw(E(K))=\emptyset$, so $|cw(E(K))|=0$. If $|cw(E(K))|=1$, we say that $K$ is \emph{almost-fibered}.

F. Manjarrez-Guti\'errez \cite{Fabiola} examined knot exteriors and showed the following results:

\begin{teo} [F. Manjarrez-Guti\'errez]
Let $K\subset S^{3}$ be a knot. At least one of the following holds:
\begin{itemize}
\item $K$ is fibered;
\item $K$ is almost-fibered;
\item $K$ contains a closed essential surface in its complement. Moreover, this closed essential surface is in the complement of an incompressible Seifert surface of $K$;
\item $K$ has at least two non-isotopic, incompressible Seifert surfaces.
\end{itemize}
\end{teo}

\begin{teo} [F. Manjarrez-Guti\'errez]\label{Inc}
If $(E(K),\mathcal{D})$ is in circular thin position, then
\begin{itemize}
\item Each Heegaard splitting $S_{i}$ of $W_{i}$ is strongly irreducible.
\item Each $F_{i}$ is incompressible in $E(K)$.
\item Each $S_{i}$ is weakly incompressible surface in $E(K)$.
\end{itemize}
\end{teo}

The converse of this theorem is not true in general. That is, a circular handle decomposition satisfying the three properties above need not be thin.

\subsection{Companion annuli in genus two handlebodies.} 
In this section much of the definitions and notations follow from L. Valdez-S\'anchez \cite{Tejano}.

Let $M$ a $3$-manifold with boundary and $\gamma\subset\partial M$ a circle which is non-trivial in $M$. We say that a separating annulus $A$ properly embedded in $M$ is a \emph{companion annulus of} $\gamma$ if $A$ is not parallel into $\partial M$ and the circle components of $\partial A$ cobound an annulus $A_{\gamma}\subset\partial M$ with core isotopic to $\gamma$ in $\partial M$. If the region cobounded by $A$ and $A_{\gamma}$ in $M$ is a solid torus $V$, we say that $V$ is a \emph{companion solid torus of} $\gamma$ in $M$ and denote the components of $M\backslash A$ by $M_{A}$ and $V$.

%\begin{lema}[\cite{Tejano}]\label{tejano1}
%Let $M$ be an irreducible $3$-manifold with boundary and $\gamma\subset\partial M$ a separating circle that is not trivial in $M$ such that $\partial M=T_{1}\cup_{\gamma}F$, where $T_{1}\subset\partial M$ is a once-punctured torus. Then $T_{1}$ is incompressible in $M$ and there is, up to isotopy, at most one circle in $T_{1}$ which has a companion annulus in $M$.
%\end{lema}

Let $H$ be a genus two handlebody and $\gamma ,\gamma'\subset\partial H$ be mutually disjoint and non-parallel circles. We say that
\begin{itemize}
\item $\gamma$ is a \emph{primitive circle} in $H$ if $\gamma$ represents a primitive element in the free group $\pi_{1}(H)$; geometrically, this is equivalent to the presence of a disk in $H$ which intersects $\gamma$ minimally in one point.
\item $\gamma$ is a \emph{power circle} in $H$ if $\gamma$ represents a non-trivial power in $\pi_{1}(H)$, that is, if $\gamma$ represents a power $p\geq 2$ of some non-trivial element in $\pi_{1}(H)$.
%\item %$\gamma,\gamma'\subset\partial H$ are \emph{coannular} if they cobound an annulus in $H$, and \emph{separated} if there is a separating non-trivial disk (a \emph{waist} disk) in $H$ separating $\gamma$ and $\gamma'$
%\item %$\gamma ,\gamma'\subset\partial H$ are \emph{basic circles} in $H$ if they represent a basis of the group $\pi_{1}(H)$ (relative to some base point), in which case, by the 2-handle addition theorem $\gamma$ and $\gamma'$ must be separated circles.
\end{itemize}
In the part $2$ of Lemma $3.3$ of \cite{Tejano}, Valdez-S\'anchez prove the next

\begin{lema}\label{tejano2}
Let $\gamma\subset\partial H$ a circle which is non-trivial in $H$. Then $\gamma$  has a companion annulus in $H$ iff $\gamma$ is a power circle in $H$; more precisely,
\begin{itemize}
\item the companion annulus $A$ of $\gamma$ is unique up to isotopy and cobounds with $\partial H$ a companion solid torus of $\gamma$, of whose core $\gamma$ represents a non-trivial power in $\pi_{1}(H)$.
\item $H\backslash A$ consists of  a genus two handlebody $H_{A}$ and a solid torus, and the core of $A$ is a primitive circle in $H_{A}$.
%the surface $\partial H\backslash\gamma$ compresses in $H$ iff $\gamma$ is a primitive or a power circle in $H$, in which case the following conditions hold
%\begin{enumerate}%$\partial H\backslash\gamma$ compresses along a waist disk $D_{w}\subset H$ which cuts $H$ into two solid tori $V,V'\subset H$ with $\gamma\subset\partial V$.
%$\partial H\backslash\gamma$ compresses along a non-separating disk in $H$, which is unique up to isotopy;

\end{itemize}
%\end{enumerate}
\end{lema}

%\begin{lema}[\cite{Tejano}]\label{tejano3}
%Let $H$ be a genus two handlebody and $\gamma, \gamma'\subset\partial H$ a pair of disjoint circles.
%\begin{enumerate}
%\item If $M=H\cup_{\gamma}V$ is a manifold obtained by gluing a solid torus $V$ to $H$ along an annular neighborhood $A=\partial H\cap\partial V$ of $\gamma$, such that $A$ runs at least twice around $V$, then $M$ is a genus two handlebody iff $\gamma$ is a primitive circle in $H$.
%\item If $M=V\cup_{\gamma}H\cup_{\gamma'}V'$ is a manifold obtained by gluing solid tori $V,V'$ to $H$ along disjoint annular neighborhoods $A=\partial H\cap\partial V$ and $A'=\partial H\cap\partial V'$ of $\gamma$ and $\gamma'$, respectively, where each annulus $A,A'$ runs at least twice around $V,V'$, respectively, then $M$ is a genus two handlebody iff $\gamma$ and $\gamma'$ are primitive circles in $H$.
%\end{enumerate}
%\end{lema} 

\section{Knots with four genus one Seifert surfaces}\label{ejemplos}
\(\)
Suppose that $K\subset S^3$ is a genus one knot and $\mathbb{T}=T_1\sqcup T_2\sqcup T_3\sqcup T_4$ is a collection of four mutually disjoint and non-parallel genus one Seifert surfaces for $K$, where the $T_i$'s are labeled consecutively around $\partial\mathcal{N}(K)$ following some fixed orientation of the meridian slope $\mu\subset\partial\mathcal{N}(K)$. 
For $i\in\{1,\ldots, 4\}$, denote by $H_i$ the region in $E(K)$ cobounded by $T_i$ and $T_{i+1}$.
For each $i=1,\ldots, 4$, set $\partial H_i=T_i\cup {J_i}\cup T_{i+1}$ (mod $4$) with $J_i\subset\partial\mathcal{N}(K)$ an annulus whose core is parallel to $K$. See Figure \ref{descK}.

\begin{figure}
\begin{center}
\includegraphics[scale=.4]{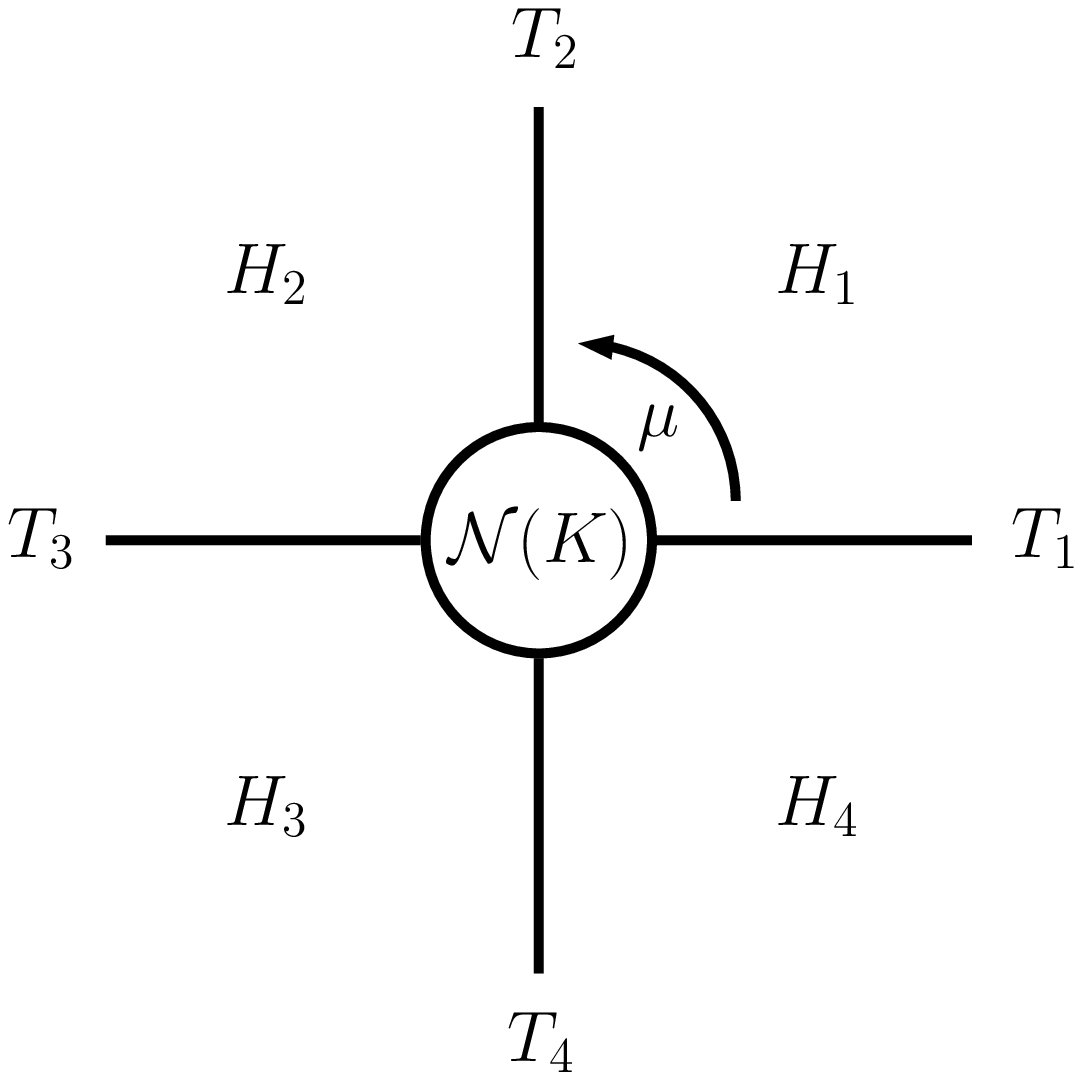}    
\end{center}
\caption{}
\label{descK}

\end{figure}

Assume that  $H_i$ satisfy the following conditions:

\begin{itemize}
    \item Each of the regions $H_i$ is a genus two handlebody.
    \item There is a power circle $\omega_i\subset T_i$ in $H_i$ with companion annulus $C_i\subset H_i$, such that the circles $\partial C_i$ cobound an annulus $\overline {C}_i\subset T_i$, and such that the once-punctured torus obtained by  isotoping $(T_i\backslash \overline{C}_i)\cup C_i$ slightly off $T_i$ is parallel to $T_{i+1}$. 
    \item $\omega_i$ represents a power $p\geq 3$ of some non-trivial element in $\pi_{1}(H)$.
    \item Let $\omega_i'$ be a parallel copy of the core of the annulus $C_i$ in $T_{i+1}$. Then $\omega_i'$ and $\omega_{i+1}$ intersect transversely in one point.
\end{itemize}

Note that the curve $\omega_i'$ in $T_{i+1}$ is a power circle in $H_i$, with companion annulus $C_i'$, such that $C_i\cap C_i'$ intersect in two simple closed curves.

\begin{teo} \label{familyknots} There is an infinite family of genus one knots $K(\ell,m,n)$ that satisfy all the above properties. \end{teo}

\begin{proof} Let $K'\cup L$ be the link shown in Figure \ref{generalizacion}, where $\ell,m,n$ are integers. The  box labelled by $m$ stands for a sequence of $m$ horizontal crossings. Assume that $\vert m \vert, \vert n \vert, \vert \ell \vert \geq 3$. Note that both $K'$ and $L$ are trivial knots. 

\begin{figure}
\begin{center}
\includegraphics[scale=.35]{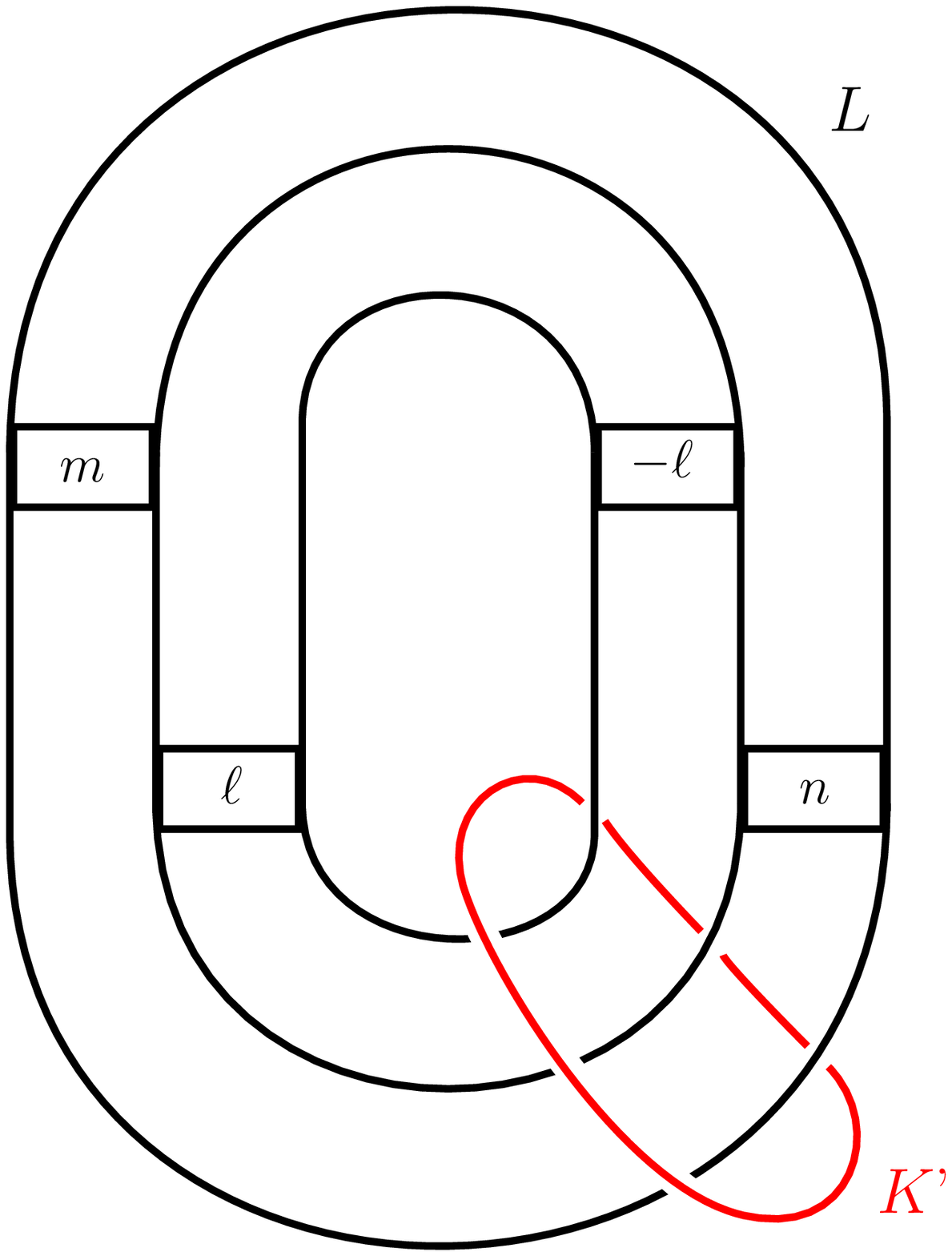}    
\end{center}
\caption{}
\label{generalizacion}

\end{figure}

Let $D_1$, $D_2$, $D_3$ and $D_4$ be the disks shown in Figure \ref{discos}, that is, disks properly embedded in $E(K')$, their boundary lying in $\partial \mathcal{N}(K')$. Note that $\partial D_i =K'$, and that $D_i$ intersects $L$ transversely in 3 points.

\begin{figure}
\begin{center}
\includegraphics[scale=.34]{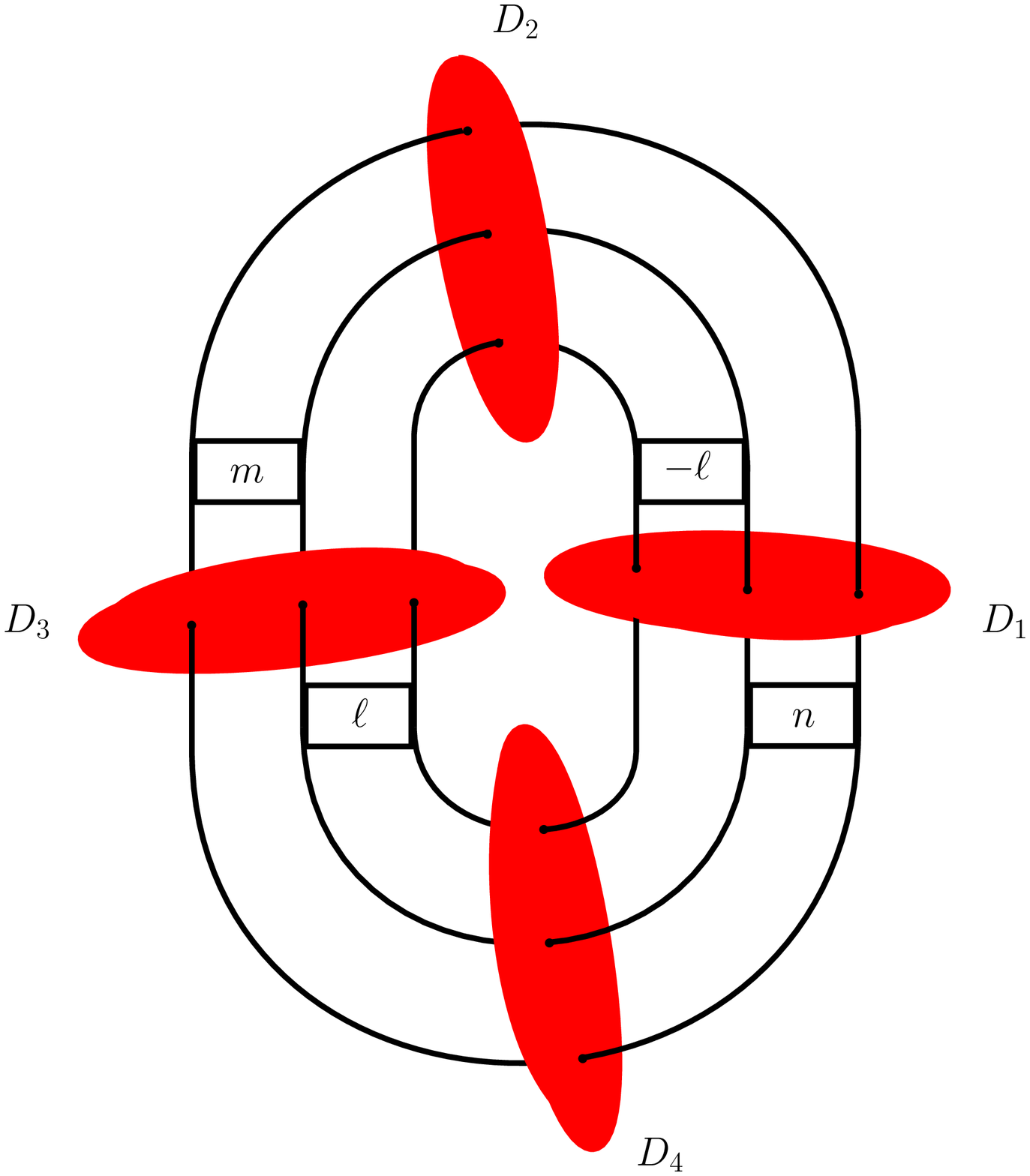}  \end{center}

\caption{}
\label{discos}

\end{figure}

Let $p: S^3 \rightarrow S^3$ be the double cover of $S^3$ branched along $L$. The lifting $K=K(\ell,m,n)=p^{-1}(K')$  is a knot in $S^3$ since the linking number between $K'$ and $L$ is an odd number. Let $T_i=p^{-1}(D_i)$, $i=1,2,3,4$. Then $T_i$ is a genus one Seifert surface for $K$. Let $H_i'$ be the region bounded by $D_i$ and $D_{i+1}$. This is a 3-ball that intersects $L$ in 3 arcs. Note that considered as a 3-tangle, it is a trivial tangle. Therefore,
$H_i=p^{-1}(H_i')$ is a genus two handlebody. Let $E_i'$ be a disk in $D_i$ that contains two points of intersection with $L$, and such that there is another 
disk $E_i$ in $H_i'$, with $\partial E_i' = \partial E_i$, and $E_i' \cup E_i$ bounds a 3-ball $B_i$ contained in $H'_i$ which encloses the 
crossings labelled $\ell$, $m$, or $n$. Let $C_i'$ and $C_i$ be the liftings of $E_i'$ and $E_i$, respectively, these are annuli. The lifting of $B_i$ is a solid torus $R_i$ whose boundary is the union of $C_i'$ and $C_i$. Let $\omega_i$ be the core of $C_i'$, this is a curve that wraps around $R_i$ at least thrice (in fact it wraps around $R_i$, $\ell$, $m$ or $n$ times). Then $\omega_i$ is a power curve in $H_i$, and $C_i$ is a companion annulus. Clearly,
$(T_i\backslash C_i')\cup C_i$ is parallel to $T_{i+1}$. The curve $\omega_i$ is the lift of a curve on $E_i'$ parallel to its boundary. Note that a copy of this curve, and the corresponding curve in $E_{i+1}$ intersect in two points. However, these curves lift to a pair of curves, and in fact, each pair of curves intersect in one point. \end{proof}

\begin{teo} \label{familywidth}
The exterior of the knots $K(\ell,m,n)$ admits a circular decomposition of width $\{3,3\}$.
\end{teo}

\begin{proof} The link $K'\cup L$ shown in Figure \ref{generalizacion} is isotopic to the link shown in Figure \ref{generalizadoMorse}. Let $D_2$, $F_1$, $D_4$ and $F_2$ be the disks shown in Figure \ref{generalizadoMorse}, these are disks properly embedded in $E(K')$, with their boundary lying in $\partial \mathcal{N}(K')$. Note that $F_1$ and $F_2$ intersect $L$ transversely in 5 points. As before, $T_i=p^{-1}(D_i)$, $i=2,4$, is a genus one Seifert surface for $K=K(\ell,m,n)$, while $\Sigma_i=p^{-1}(F_i)$, $i=1,2$, is a genus two Seifert surface for $K$.
By taking a collection of parallel disks between $D_2$ and $F_1$, we note that there is  collection of disks intersecting $L$ in three points, then there is a single disk intersecting $L$ in three points and in a local maximum point and then a collection of disks intersecting $L$ in five points. In the double branched cover, this implies that we pass from $T_2$ to $\Sigma_1$ by adding a 1-handle. Similarly, when passing from $F_1$ to $D_4$, there is a disk that intersects $L$ in a local minimum. In the double branched cover this implies that we pass from $\Sigma_1$ to  $T_4$ by adding a 2-handle. Something similar can be done for $D_4$, $\Sigma_2$ and $D_2$. Then $T_2$, $\Sigma_1$, $T_4$ and $\Sigma_2$ define a circular decomposition for the exterior of $K$, which has width $\{ 3,3 \}$. \end{proof}

\begin{figure}
\begin{center}
\includegraphics[scale=.34]{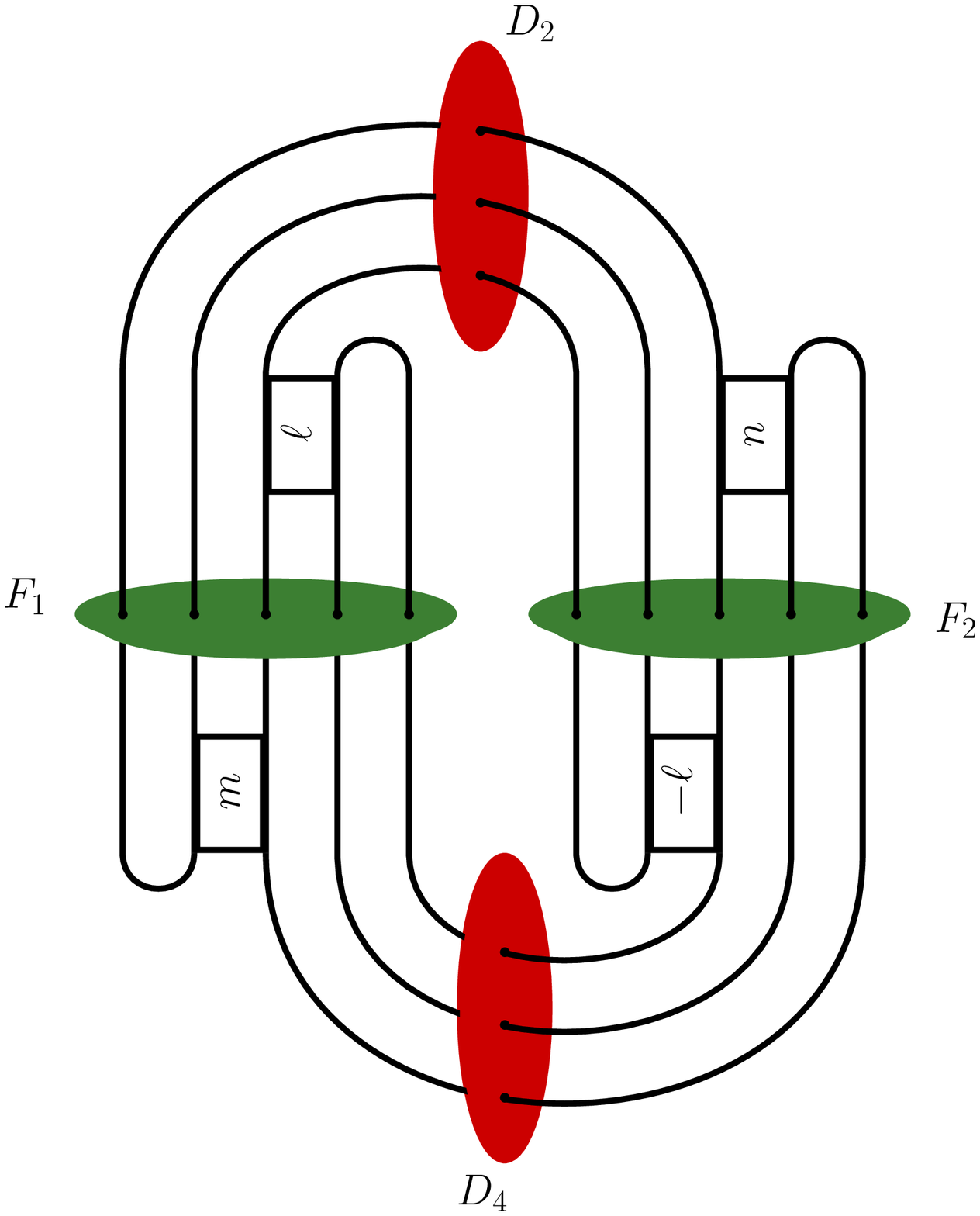}    
\end{center}
\caption{}
\label{generalizadoMorse}

\end{figure}

\section{Knots without a decomposition of width $\{ 3 \}$}\label{width}
\(\)

Suppose that $K\subset S^3$ is a genus one knot and $\mathbb{T}=T_1\sqcup T_2\sqcup T_3\sqcup T_4$ is a collection of four mutually disjoint and non-parallel genus one Seifert surfaces for $K$. Denote by $H_1$,  $H_2$, $H_3$, and $H_4$ the complementary regions of the surfaces in the exterior of $K$. Suppose that all conditions for the surfaces and the regions stated just before Theorem \ref{familyknots} are satisfied.

\begin{lema} \label{hyperbolic}
The knot $K$ is hyperbolic. \end{lema}
\begin{proof}
The given conditions imply that for each $T_i$ ($i=1,\ldots, 4$) a power circle in $H_i$ cannot have companion annuli on both sides of $T_i$.
Then part 1 of Lemma 8.1 of \cite{Tejano} implies that $H$ is hyperbolic.
\end{proof}

%\begin{lema}\label{5anillos}
%For each $i=1,\ldots, 4$, there are at most five essential and non-parallel annuli properly embedded in $H_i$ and whose boundaries lie in $T_i \sqcup T_{i+1}$. 
%\end{lema}

\begin{lema}\label{5anillos}
For each $i=1,\ldots, 4$, there are up to isotopy five incompressible annuli properly embedded in $H_i$, whose boundaries lie in $T_i \sqcup T_{i+1}$, and which are not parallel onto $T_i$ or $T_{i+1}$. 
\end{lema}
\begin{proof}

First we note that $H_i$ contains five incompressible annuli properly embedded associated to the power circle in $H_i$. These annuli are: 
\begin{itemize}
    \item The companion annulus $C_i$ for the power circle in $H_i$ such that $\partial C_i\subset T_i$.
    \item The companion annulus $C_i'$ for the power circle in $H_i$ such that $\partial C_i'\subset T_{i+1}$.
    \item The spanning annulus $A_i$ for the power circle in $H_i$ with one boundary component in $T_i$ and the other in $T_{i+1}$.
    \item A second spanning annulus $A_i'$ for the power circle in $H_i$  with one boundary component in $T_i$ and the other in $T_{i+1}$.
    \item An annulus $\widehat{A_i}$ boundary parallel into $\partial\mathcal{N}(K)$.
\end{itemize}

The spanning annuli $A_i$ and $A_i'$ are obtained by doing one of the two possible oriented double curve sums of $C_i$ and $C_i'$.
In the complementary region $H_i$ we represent these five annuli as in Figure \ref{anillos}. That is, we represent by a blue circle the companion torus, by green arcs the companion annuli, by red arcs the spanning annuli, and by a purple arc the annulus parallel into $\partial\mathcal{N}(K)$.

Note that each one of these annuli is not boundary parallel into $T_i$ or $T_{i+1}$, and that any $\partial$-compression disk for these annuli must intersect the annulus $J_i$ contained in $\partial \mathcal{N}K$.

\begin{figure}
\begin{center}
\includegraphics[scale=.5]{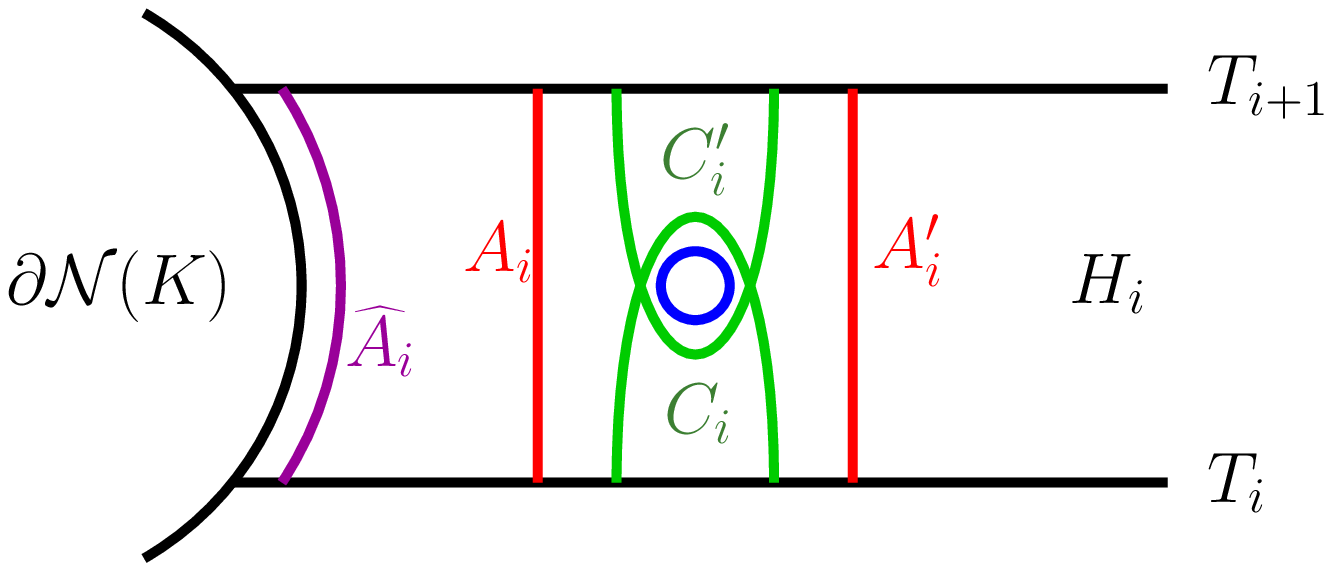}    
\end{center}
\caption{}
\label{anillos}
\end{figure}

Now suppose that $B$ is a properly embedded and incompressible annulus in $H_i$, whose boundary lies in $T_i\sqcup T_{i+1}$. If there is a $\partial$-compression disk for $B$ disjoint from $J_i$, then $B$ is boundary parallel, for otherwise $T_i$ or $T_{i+1}$ would be compressible. Isotope $B$ to intersect $C_i\sqcup C_i'\sqcup A_i\sqcup A_i'\sqcup  \widehat{A_i}$ minimally. Trivial closed intersection curves on $B$ or $C_i\sqcup C_i'\sqcup A_i\sqcup A_i'\sqcup \widehat{A_i}$ can be eliminated since the surfaces are incompressible and $H_i$ is irreducible. Suppose first that $\partial B$ is not disjoint from the boundary curves of all the annuli $C_i\sqcup C_i'\sqcup A_i\sqcup A_i'\sqcup \widehat{A_i}$. Then there must be intersection arcs between them. If there is an arc of intersection which is trivial in one of the annuli, then either it can be removed by an isotopy or it would imply that one of the annuli is $\partial$-compressible, which is not possible. If $\partial B$ has both components on $T_i$, then there must be trivial intersection arcs on $A_i$, and if $\partial B$ has one component in $T_i$ and the other in $T_{i+1}$, then there is a trivial intersection arc in $B$. 

Then suppose that $\partial B$ is disjoint from the boundary curves of all the annuli $C_i\sqcup C_i'\sqcup A_i\sqcup A_i'\sqcup \widehat{A_i}$. If there are closed essential intersection curves on $B$ or $C_i\sqcup C_i'\sqcup A_i\sqcup A_i'\sqcup \widehat{A_i}$, then the boundary curves of $B$ would be parallel to the boundary curves of one of the annuli $C_i, C_i', A_i, A_i', \widehat{A_i}$. By taking an outermost curve of intersection lying in $B$, an isotopy can be done to reduce the number of curves of intersection. Thus we have that $B\cap(C_i\sqcup C_i'\sqcup A_i\sqcup A_i'\sqcup \widehat{A_i})=\emptyset$. Suppose $B$ is not parallel into $\partial H_i$. By symmetry we may assume only two cases: $\partial B\subset T_i$ (or $T_{i+1}$) and $\partial_1 B\subset T_i$, $\partial_2 B\subset T_{i+1}$.
If $\partial B\subset T_i$, then $\partial B$ cobound an annulus on $T_i$ and $B$ is separating on $H_i$. Then $B$ is a companion annulus, but there is up to isotopy only one companion annulus in $H_i$ (Lemma \ref{tejano2}), so $B$ is parallel to $C_i$. Similarly, if $\partial B\subset T_{i+1}$, then $B$ is parallel to $C'_i$.
Now, suppose $\partial_1 B\subset T_i$, $\partial_2 B\subset T_{i+1}$. Since $B\cap(C_i\sqcup C_i'\sqcup A_i\sqcup A_i'\sqcup \widehat{A_i})=\emptyset$, $\partial_1 B\subset T_i$ must be parallel to either $\partial C_i$ or $\partial_1\widehat{A_i}$ and $\partial_2 B\subset T_{i+1}$ must be parallel to either $\partial C'_i$ or $\partial_2\widehat{A_i}$ . Then $B$ is parallel to one of the spanning annuli $A_i, A_i'$, or $\widehat{A_i}$ .

\end{proof}

\begin{lema}\label{noanillosconsecutivos}
If $A$ is an incompressible annulus properly embedded in $H_i\cup H_{i+1}$, whose boundary lies on $T_i\cup T_{i+2}$, then either it is parallel to an annulus in $\partial \mathcal{N}(K)$, or it can be isotoped to lie in $H_i$ or $H_{i+1}$. \end{lema}

\begin{proof} This follows from Lemma \ref{5anillos} and from the fact that there are no companion annuli on both sides of $T_{i+1}$ with the same boundary slope (Lemma 5.1 of \cite{Tejano}). \end{proof}

\begin{lema}\label{3pK}
For each $i=1,\ldots, 4$, $H_i$ cannot contain a properly embedded pair of pants $P$ with all its boundaries parallel to $K$, that is, such that each component of $\partial{P}$ lies on $T_i$ (or $T_{i+1}$) and it is parallel to $\partial T_i$ (or $\partial T_{i+1}$), or it is an essential curve on $J_i \subset \partial \mathcal{N}(K)$.
\end{lema}
\begin{proof}
Suppose $H_i$ contains a properly embedded pair of pants $P$, with boundaries all parallel to $K$. Up to isotopy, we can assume that  $\partial P$ is contained in $J_i$. Then there is planar surface with three boundary components properly embedded in $E(K)$, which is not possible by \cite{Gabai}, unless $P$ is compressible. But if this happens, by compressing $P$ we get that $K$ is the trivial knot, which is not possible.
\end{proof}

\begin{lema}\label{pants}
Suppose that $P$ is an incompressible pair of pants properly embedded in $H_i$, whose boundary is contained in $T_i \sqcup T_{i+1}$. Suppose that two of its boundary components are essential, non-boundary parallel, and non power curves on $T_{i+1}$. Suppose that the third boundary component lies on $T_i$. Then this last curve is parallel on $T_i$ to $\partial T_i$.
\end{lema}

\begin{proof}
Consider the intersection between $P$ and the companion $C_i'$. Note that $P$ and $C_i'$ must intersect, and the intersection consists of arcs, for simple closed curves of intersection can be removed by the incompressibility of the surfaces. Suppose that $\gamma$ is an arc of intersection in $C_i'$, with endpoints lying in the same boundary component of $\partial C_i'$, and that $\gamma$ is outermost. So, $\gamma$ cuts off a disk $D$ in $C_i'$ which is a $\partial$-compression disk for $P$. By performing the boundary compression, we get an annulus, with one boundary component in $T_{i+1}$ which is parallel to $\partial T_{i+1}$. Then by Lemma \ref{3pK}, the other boundary component of the annulus, that is, the boundary component of $P$ lying in $T_i$, must be parallel to $\partial T_i$.

The remaining possibility is that all the arcs of intersection in $C_i'$ are essential, that is, spanning arcs in $C_i'$. There are at least two arcs of intersection, for there are at least four points of intersection between $\partial P$ and $\partial C_i'$. Take two consecutive arcs of intersection in $C_i'$, note that when considered in $P$, there are two possibilities for the arcs. A possible case is that each arc has endpoints on the same boundary component, but a different component from the other arc. In this case one of the arcs is trivial, and then there is a $\partial$-compression disk for $C_i'$, which is not possible. The other case is when both arcs join two different boundary components. Then there is a disk $D$ in $P$ formed by these two arcs plus one arc in each of the boundary components. Let $\overline{C}_i'$ be the annulus in $T_{i+1}$ bounded by $\partial C_i'$. One possibility is that $D$ is a compression disk for the torus $(T_{i+1}\backslash \overline{C}_i' )\cup C_i'$, which is not possible. The remaining possibility is that $D$ is a meridian disk of the companion solid torus of $H_i$, but this implies that the power curve in $H_i$ is a power of order two, which by hypothesis is not possible.
\end{proof}

\begin{lema}\label{4S}
The exterior of $K$, $E(K)$, cannot contain another genus one Seifert surface for $K$ non isotopic to any $T_{i}$; $i=1,\ldots, 4$.
\end{lema}
\begin{proof}
Suppose that $E(K)$ contains a fifth genus one Seifert surface for $K$, $S$, non isotopic to any $T_i$ ($i=1,\ldots, 4$). Isotope $S$ in order to make $\partial S$ and $\partial T_i$ disjoint for every $i=1,\ldots, 4$. 

If $S\cap\mathbb{T}=\emptyset$, then $S\subset H_i$ for some $i\in\{1,\ldots,4\}$. By Lemma $6.8$ of \cite{Tejano}, $S$ is parallel to $T_i$ or $T_{i+1}$.

Now, if $S\cap\mathbb{T}\neq\emptyset$, by the incompressibility of $S$ and $T_i$ ($i=1,\ldots, 4$) we can isotope $S$ to intersect transversely $\mathbb{T}$ in a minimal set of essential closed curves (in both $S$ and $T_i$; $i=1,\ldots, 4$). $\mathbb{T}$ divide $S$ into a collection of annuli, pair of pants, and once punctured tori. Suppose $\partial S \subset J_i \subset H_i$ for some $i\in\{1,\ldots,4\}$. 

%We can not have a once punctured tori in $S$ contained in a region $H_i$ because it would be parallel to $T_i$ or $T_{i+1}$. On the other hand, 
If there is an annulus $A$ in $S$ with one boundary $\partial S$ and the other, $\partial_2 A$ in $T_{i+1}$, then $\partial_2 A$ is parallel to $K$ and we can isotope $S$ to eliminate such intersection. Then the curves of intersection in $S$, are not parallel to $\partial S$. If there is a curve of intersection between $T_i$ and $S$, which is parallel to $\partial T_i$ in $T_i$, then by doing an annulus compression to $S$, we get a pair of pants with its boundaries contained in $\partial \mathcal{N}(K)$, which is not possible, by the same argument as in Lemma \ref{3pK}. 
We can assume that $\mathbb{T}$ divide $S$ into a collection of annuli and a pair of pants. Note that in this collection, we can not have adjacent annuli, by Lemma \ref{noanillosconsecutivos}. 

In this way, $\mathbb{T}$ divide $S$ into an annulus $A$ and a pair of pants $P$. Since the annulus $A$ is parallel to the companion annulus, say in $H_i$, its boundary curves are the  power circle in $H_i$ and assuming that the pair of pants $P$ is contained in $H_{i-1}$, we have that in $H_{i-1}$, $P$ has two boundary curves in $T_i$ which are not power circles $H_{i-1}$ and one boundary curve parallel to $K$. It follows that $P$ is parallel to $T_{i}$ minus an annulus, then $S$ can be isotoped to be contained in $H_i$ and to be parallel to $T_{i+1}$, a contradiction.    
\end{proof}

\begin{lema}\label{morepants}
Suppose that $P$ is an incompressible pair of pants properly embedded in $H_i$, whose boundary is contained in $T_i$. Suppose that two of its boundary components are essential, non-boundary parallel on $T_i$, and that the third boundary component is parallel on $T_i$ to $\partial T_i$. Then either $P$ is parallel into $T_i$, or two boundary components of $\partial P$ are isotopic to the power curve of $H_i$, and $P$ union an annulus in $T_i$ is parallel into $T_{i+1}$.
\end{lema}

\begin{proof} If two boundary components of $P$ are not the power curve of $H_i$, then $P$ must intersect the annulus $A_i$. By looking at an outermost arc of intersection in $A_i$ we get that $P$ is $\partial$-compressible. By $\partial$-compressing $P$ we get an annulus parallel to $T_i$, and hence $P$ must be parallel into $T_i$. If two boundary components of $P$ are isotopic to the power curve of $H_i$. These curves bound an annulus $C$ on $T_i$, such that $P\cup C$ is a once-puncured tori, which by Lemma \ref{4S} is either parallel to $T_i$, or to $T_{i+1}$.
\end{proof}

\begin{teo}\label{principal}
$E(K)$ cannot admit a circular decomposition of width $\{3\}$.
\end{teo}

\begin{proof}
We are assuming that $K$ is a knot whose exterior contains $4$ mutually disjoint and non-parallel genus one Seifert surfaces $T_{i}$; $i=1,\ldots, 4$, and that $\mathbb{T}=\sqcup_{i=1}^4 T_i$.
Suppose $E(K)$ has a circular decomposition of width $\{3\}$. So, 
$$
E(K)=(T_{1}\times I)\cup N\cup B/ T_{1}\times 0\sim T_{1}\times 1
$$ where $N$ and $B$ stand for a one handle and a two handle, respectively.

By Lemma \ref{4S}, $T_{i}$ ($i=1,\ldots, 4$) are the only genus one Seifert surfaces for $K$. 
 
 Let $\Sigma:=\partial ((T_{1}\times I)\cup N)$. Then  $\Sigma$ is a genus two Seifert surface for $K$. By Theorem \ref{Inc}, $\Sigma$ is weakly incompressible.
 
 We note that $\Sigma\cap T_{i}\neq\emptyset$ for $i\in\{2,3,4\}$. To show this, suppose that  $\Sigma\cap T_{4}=\emptyset$. Let $D$ the core disk of $B$. By the incompressibility of $T_{4}$, we can isotope $D$ to be disjoint of $T_{4}$. By compressing $\Sigma$ along $D$ we obtain a genus one surface which does not intersect the region $H_4$ and then cannot be isotopic to $T_{1}$, a contradiction.
 
 Without lost of generality,  suppose that $\partial\Sigma\subset H_{1}$ or $\partial\Sigma\subset H_{2}$. By general position, $\Sigma\cap T_{i}$ for $i\in\{2,3,4\}$ consists of simple closed curves. Moreover, as $\Sigma$ is weakly incompressible, $\Sigma\cap T_{i}$ consists of simple closed curves which are essential on both $\Sigma$ and $T_{i}$, by Lemma 2.6 of \cite{Sch}. We assume that $\Sigma$ has been isotoped such that the number of connected components of $\Sigma \backslash \mathbb{T}$, denoted $c(\Sigma)$, is minimal.
 $\mathbb{T}$ divides $\Sigma$ into a collection of  annuli, pair of pants, four punctured spheres, once punctured tori, and twice punctured tori. Let $\Sigma_i=\Sigma \cap H_i$ for $i\in\{1,2,3,4\}$.

If there is an annulus in $\Sigma$ with one boundary in $\partial\mathcal{N}(K)$ and the other in $T_{2}$ or $T_3$, isotope $\Sigma$ to eliminate such intersection. Also, by Lemma \ref{noanillosconsecutivos}, there cannot be adjacent annuli on $\Sigma$ with boundary slope a power circle in $H_{i}$.

\begin{lema} No component of $\Sigma_4$ is a 4-punctured sphere.
\end{lema}

\begin{proof} Suppose that a component of $\Sigma_4$ is a 4-punctured sphere $S_4$. Then $\partial S_4 \subset T_4$. As $\Sigma$ is a genus 2 surface, the complementary regions of $\Sigma_4$ in $\Sigma$ are an annulus $A$ and a pair of pants $P$. Note that $A$ must be contained in $H_3$, and it is a companion annulus there,  otherwise we can push it into $H_4$, reducing $c(\Sigma)$. $P$ is contained in the union of $H_3$, $H_2$ and $H_1$. $\partial P$ consists of $\partial \Sigma$ and two curves $\alpha _1$, $\alpha_2$. No curve of intersection of $P$ with $T_3$ and $T_2$ is parallel to $\partial \Sigma$. So $P\cap (T_3 \cup T_2)$ consists of curves parallel to $\alpha_1$ or $\alpha_2$. This implies that $P\cap H_3$ consists of annuli, and also $P\cap H_2$ consists of annuli. As we have two adjacent annuli these must be parallel to a neighborhood of the knot. That is, $\alpha_1$ is a curve parallel to $K$. If one of these annuli have its boundary components in the same $T_i$, an isotopy removes the intersection. If all annuli are spanning, then in $H_1$ there is a pair of pants, whose all boundary components are parallel to $K$, which contradicts Lemma \ref{3pK}.
\end{proof}

\begin{lema} If $\Sigma_4$ contains a once punctured torus $T$ or a pair of pants, then part of it can be isotoped through $T_4$, changing it into an annulus, without changing $c(\Sigma)$. 
\end{lema}

\begin{proof} If $\Sigma_4$ contains a once punctured torus $T$, then it has its boundary on $T_4$, which is parallel to $\partial T_4$, and then it is parallel to $T_1$. Part of $T$ can be pushed into $H_3$, except for an annulus, which remains in $H_4$. This is done without changing $c(\Sigma)$.

If $\Sigma_4$ contains a pair of pants $P$, then its boundary lies on $T_4$. If only one or the three boundary components of $P$ were essential curves on $T_4$, then there would be a simple closed curve in $T_4$ intersecting $P$ and then $\Sigma$ an odd number of times, which is not possible for any such closed curve has linking number $0$ with $K$. Then two boudary components of $P$ are essential curves on $T_4$ and the third one is parallel to $\partial T_4$. Then by Lemma \ref{morepants}, either $P$ can be pushed towards $H_3$, or $P$ union an annulus in $T_4$ is parallel into $T_1$. Then, as in the previous case, part of the pair of pants can be pushed into $H_3$, except for an annulus, which remains in $H_4$.
\end{proof}

We can assume that $\Sigma_4$ contains an annulus $A_4$, which is then a companion annulus in $H_4$.

\begin{lema} The annulus $A_4$ is a non-separating annulus in $\Sigma$.
\end{lema}

\begin{proof} Suppose that $A_4$ is a separating annulus in $\Sigma$. Then it separates $\Sigma$ into a once-punctured torus $T'$ and a twice punctured torus. If $T'$ is contained in $H_3$, then by looking at the intersections between $T'$ and the annulus $C_3$, we see that there is a boundary compression disk for $T'$ contained in $C_3$, but this contradicts the so called parity rule. The other possibility is that $T'\cap H_3$ is a pair of pants and $T'\cap H_2$ is an annulus. But this annulus has to be a companion annulus in $H_2$, and then the pair of pants contradicts Lemma \ref{pants}.
\end{proof}

%Now we note that there cannot be adjacent annuli on $\Sigma$ with boundary slope a power circle in $H_{i}$. Suppose there are two adjacent annuli, say $A_{1}$ in $H_{i}$ and $A_{2}$ in $H_{i+1}$. Since $A_{1}$ and $A_{2}$ are not boundary parallel, then the boundary slope of $A_{1}$ in $T_{i+1}$ is different to the boundary slope of $A_{2}$ in $T_{i+1}$, a contradiction.

We use the symbols $A$, $P$, $S$, $T'$ to denote
an annulus, a pair of pants, a four punctured sphere and a twice punctured torus, respectively. By a symbol $Q_i$ we mean a surface of type $Q$, so that $Q_i=\Sigma\cap H_i$. Then by a sequence $A_4.P_3.S_2.A_1$ we mean that $\Sigma$ intersects the handlebodies $H_i$ in an annulus, a pair of pants, a four punctured sphere, and an annulus, respectively. If we have a sequence $Q_4.Q_3.Q_2.Q_1$, we also mean that $Q_i\cap T_i=Q_{i-1}\cap T_{i}$.

\begin{figure}
\begin{center}
\includegraphics[scale=.34]{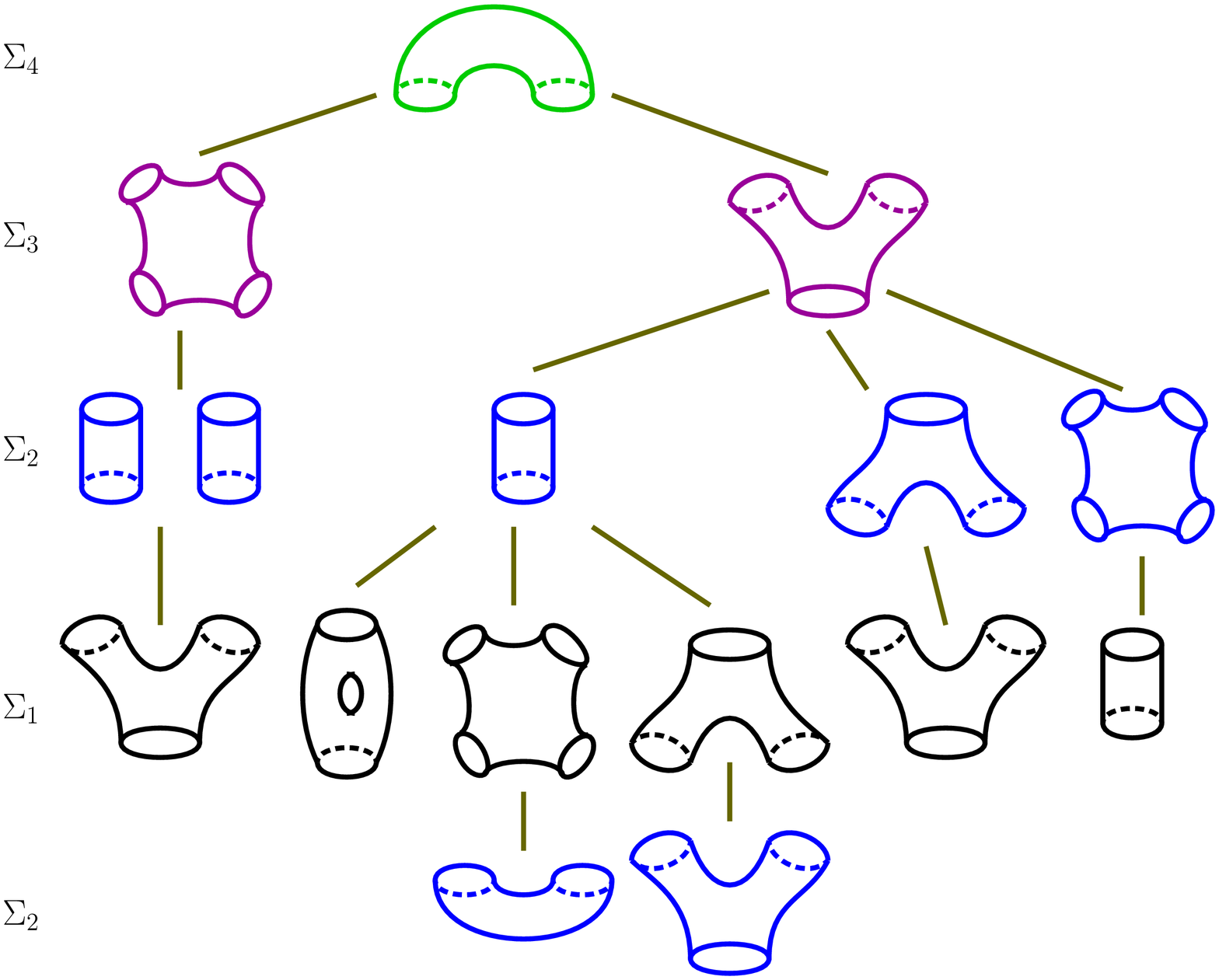}    
\end{center}
\caption{}
\label{allcases}
\end{figure}

\begin{lema} There are the following possibilities for the intersections of $\Sigma$ with the $H_i$.

\begin{enumerate}
    \item $A_4.P_3.P_2.P_1$
    \item $A_4.P_3.A_2.P_1.P_2$
    \item $A_4.S_3.\{ A_2,A_2'\}.P_1$
    \item $A_4.P_3.A_2.S_1.A_2'$
    \item $A_4.P_3.S_2.A_1$
    \item $A_4.P_3.A_2.T_1'$
    
    %\item $T_4.A_3.A_2.T_1$
   % \item $T_4.A_3.S_2.A_1$
   % \item $T_4.A_3.P_2.P_1$
   % \item $T_4.A_3.S_2.A_1$
   % \item $T_4.P_3.\{ A_2,A_2'\}.P_1$
\end{enumerate}
\end{lema}

\begin{proof} Note that $\chi(\Sigma)=\sum_{i=1}^{i=4} \chi(\Sigma_i)=-3$, and then the pieces of $\Sigma$ consist of 3 pair of pants, or a 4-punctured sphere and a pair of pants, or a twice punctured torus and a pair of pants, plus several annuli in each case. It is shown by inspection that the cases listed above are all possible cases. A schematic figure of all possible cases is shown in Figure \ref{allcases}. 

We have that $\Sigma_4$ is a non-separating annulus, which is a companion annulus in $H_4$, the region adjacent to it and contained in $H_3$, can not contain an annulus, by Lemma \ref{noanillosconsecutivos}. So, $\Sigma_3$ is either a pair of pants $P_3$ or a 4-punctured sphere $S_3$. Suppose first that $\Sigma_3=S_3$. What is left of $\Sigma$ is a pair of pants; note that $\Sigma_1$ can not contain an annulus containing $\partial \Sigma$, and then the only possibility is that $\Sigma_2$ consists of two annuli $A_2$ and $A_2'$, and $\Sigma_1$ consists of a pair of pants $P_1$. So we have Case 3.

Suppose now that $\Sigma_3=P_3$. What is left of $\Sigma$ is a twice punctured torus. Note that by Lemma \ref{pants}, the boundary of $P_3$ lying in $T_3$ is parallel to $\partial T_3$. If $\Sigma_2$ contains a pair of pants $P_2$, then $\Sigma_1$ must be a pair of pants $P_1$, so we have Case 1. If $\Sigma_2$ contains a 4-punctured sphere $S_2$, then it must contain $\partial \Sigma$, and $\Sigma_1$ is an annulus $A_1$. So, we have Case 5. If $\Sigma_2$ contains an annulus $A_2$ adjacent to $P_3$, what is left of $\Sigma$ is a twice punctured torus, and there are three possibilities. First, $\Sigma_1$ is a twice punctured torus $T_1'$, this is Case 6. Second, $\Sigma_1$ consists of a pair of pants $P_1$ and $\Sigma_2$ contains a pair of pants $P_2$, this is Case 2. And third, $\Sigma_1$ consists of a 4-punctured sphere $S_1$ and $\Sigma_2$ contains one more annulus $A_2'$, this is Case 4.
 \end{proof}

To finish the proof of Theorem \ref{principal}, we analyze each one of the cases and show that it can not arise.

\vskip10pt
\textbf{Case 1: $A_4.P_3.P_2.P_1$}.

\begin{figure}
\begin{center}
\includegraphics[scale=.44]{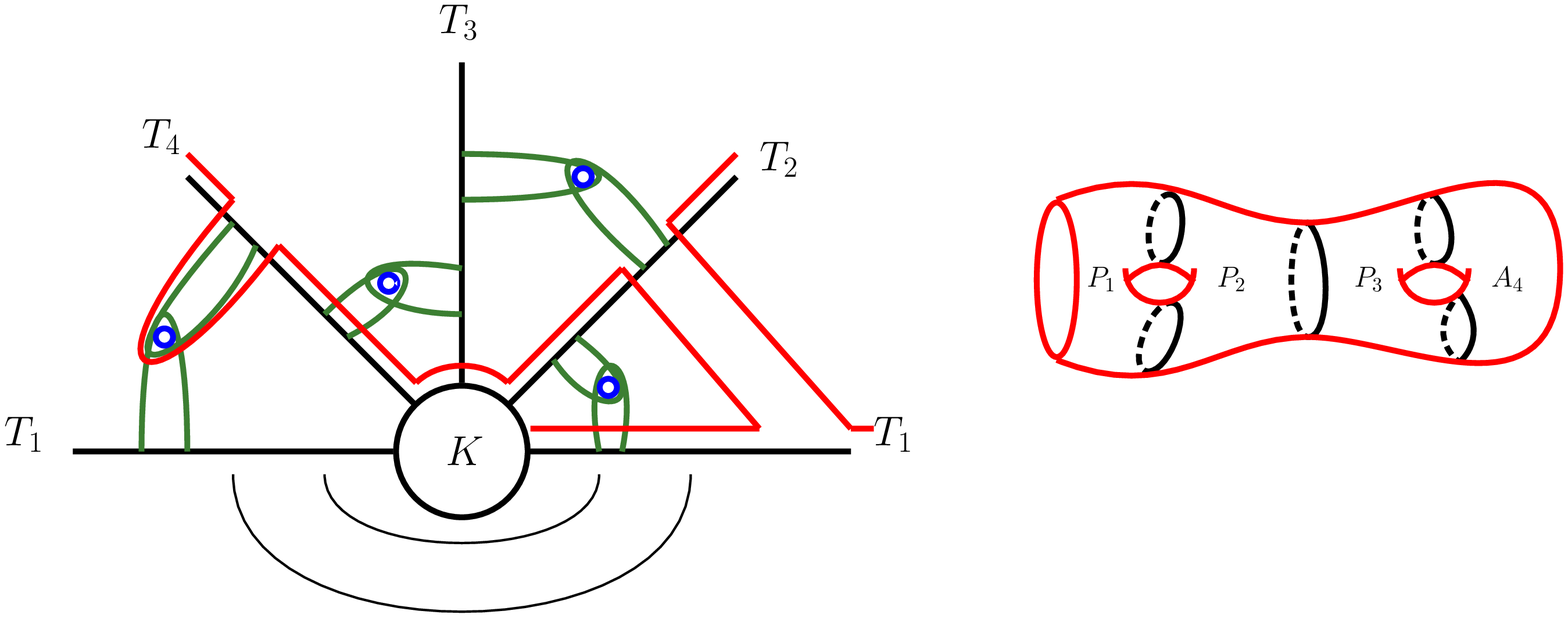}
\end{center}
\caption{}
\label{case1}
\end{figure}

In this case $A_4$ is a companion annulus in $H_4$, two of the boundary curves of $P_3$ are essential in $T_4$, and the other lies in $T_3$, and is boundary parallel, by Lemma \ref{pants}. One boundary curve of $P_2$ lies in $T_3$, and is boundary parallel there, the other two are essential curves in $T_2$. Finally, two of the boundary curves of $P_1$ are essential in $T_2$ and parallel to the boundary curves of the companion annulus in $H_1$ (for otherwise $P_1$ could be isotoped to lie in $H_2$) and the last one lies in $J_1$. This is case is schematically shown in Figure \ref{case1}. In this and the subsequent figures, the blue circles represent the companion tori, the green arcs the companion annuli, and the surface $\Sigma$ is represented by a red curve.

As $\Sigma$ is bicompressible, let $D$ be a compressing disk lying in the external side, that is, the side intersecting $\partial \mathcal{N}(K)$ in a single annulus lying in $H_1$. Consider $D\cap\mathbb{T}$.
Trivial closed intersection curves on $D$ can be eliminated by an isotopy. Then we can assume $D\cap\mathbb{T}$ consists only of arcs.

Take an innermost intersection arc $t$ on $D$, then there exists an arc $\sigma\subset\Sigma$ such that $\sigma\cup t$ cobound a disk $E\subset D$ and ${int}E\cap \mathbb{T}=\emptyset$.

\begin{itemize}

    \item If $t\subset T_4$, we have two options: $\sigma\subset H_4$ or $\sigma\subset H_3$.
    
Suppose $\sigma\subset H_4$. In $H_4$, the annulus $A_4$ is parallel to the companion annulus of $H_4$ and its boundary components lie in $T_4$. If $\sigma$ joins one boundary component of $A_4$ to itself, we can eliminate such intersection by an isotopy. Otherwise, $\sigma$ joins the two boundary components of $A_4$, giving rise to a boundary compression of this annulus, a contradiction.

Suppose now $\sigma\subset H_3$. In $H_3$, the pair of pants $P_3$ is parallel to $T_4\backslash A$ for an annulus $A\subset T_4$. Here $E$ is a compression disk for $P_3 \cup A$, which is a surface parallel to $T_4$, a contradiction.

    \item If $t\subset T_3$, we have two options: $\sigma\subset H_3$ or $\sigma\subset H_2$.

Suppose $\sigma\subset H_3$. As in the precedent case, we can think of $P_3$ in $H_3$ as $T_4\backslash A$ for an annulus $A\subset T_4$. Consider the region in $P_3$ parallel to $\partial E(K)$ between $\partial T_3$ and $\partial T_4$ which contains an annulus $A'$. In $A'$, two subarcs of $\sigma$ cobound a disk $F$ such that $E\cup F$ is an annulus. It means that $\sigma$ and $t$ can be completed to coannular closed curves $\hat{\sigma}$ and $\hat{t}$,  $\hat{\sigma}\subset T_4$ and $\hat{t}\subset T_3$. If the curve $\hat{\sigma}$ is parallel to the core of the companion annulus in $H_4$ so is $\hat{t}$, but this is not possible. Also, $\hat \sigma$ can not be parallel to the core of the companion annulus in $H_3$. Then  $\hat{\sigma}$ and $\hat{t}$ are parallel to $K$. This implies that we can eliminate the innermost arc in $D$ by an isotopy. 

Now if $\sigma\subset H_2$, we follow the same argument as  before, since $P_2$ is isotopic to $T_2\backslash A$ where a core of the annulus $A$ has the slope of the companion annulus in $H_1$.

    \item If $t\subset T_2$, we have two options: $\sigma\subset H_2$ or $\sigma\subset H_1$. 

If $\sigma\subset H_2$ we proceed as in the first case when $t\subset T_4$ and $\sigma\subset H_3$.
Suppose now $\sigma\subset H_1$. $P_1$ is a pair of pants parallel to $T_1\backslash A$, where $A$ is an annulus parallel to an annulus in $T_2$ having $t$ as an essential arc. Then $E$ is a compression disk for $P_1\cup A$ which is parallel to $T_1$, a contradiction. 
\end{itemize}

{\bf{Case 2: $A_4.P_3.A_2.P_1.P_2$}}. 

\begin{figure}
\begin{center}
\includegraphics[scale=.46]{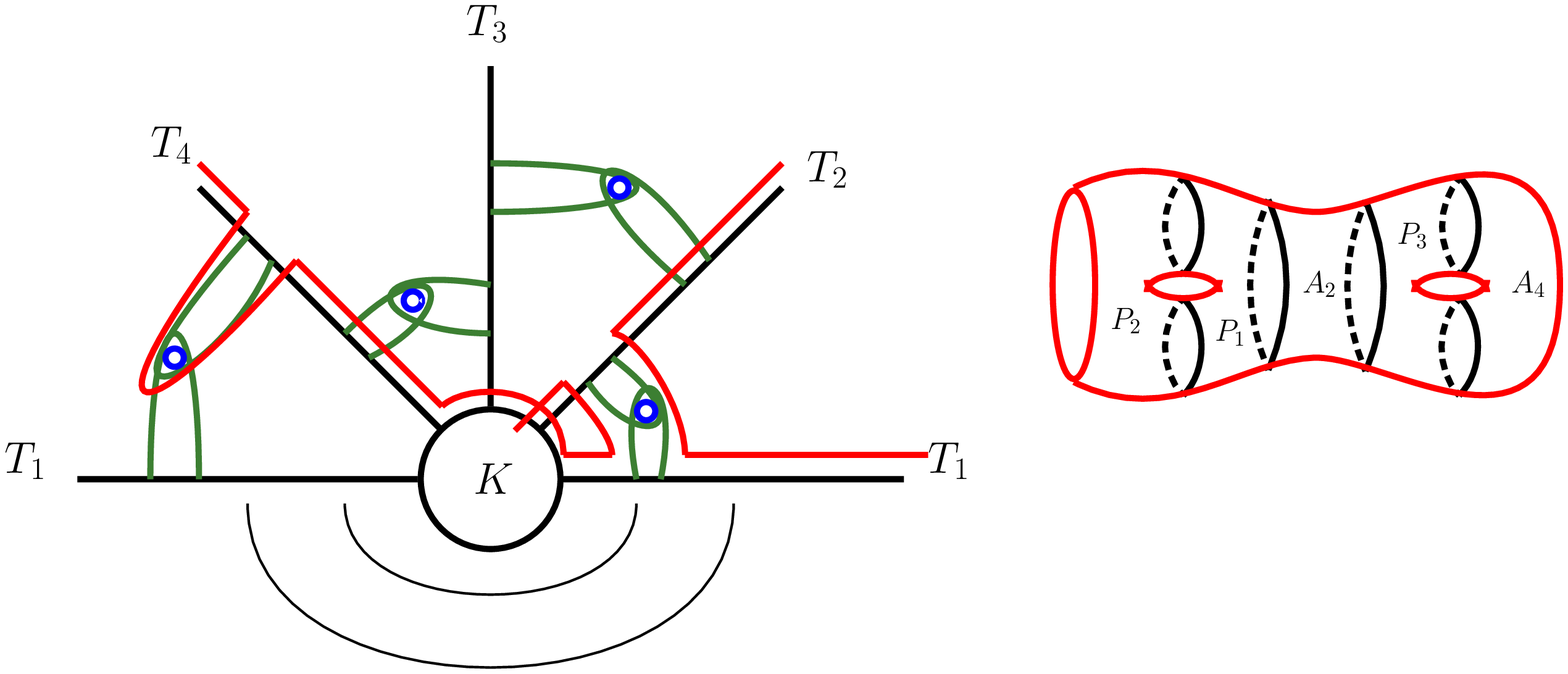}
\end{center}
\caption{}
\label{case2}
\end{figure}

This case is schematically shown in Figure \ref{case2}. In this case $\Sigma_2$ consists of an spanning annulus $A_2$ parallel to an annulus in $\partial \mathcal{N}(K)$ and a pair of pants $P_2$, whose boundary consists of two curves on $T_2$ and $\partial \Sigma$. Then $\Sigma$ would have a self-intersection, which is impossible.

\vspace{1cm}

{\bf{Case 3: $A_4.S_3.\{A_2,A_2'\}.P_1$}}.

\begin{figure}
\begin{center}
\includegraphics[scale=.46]{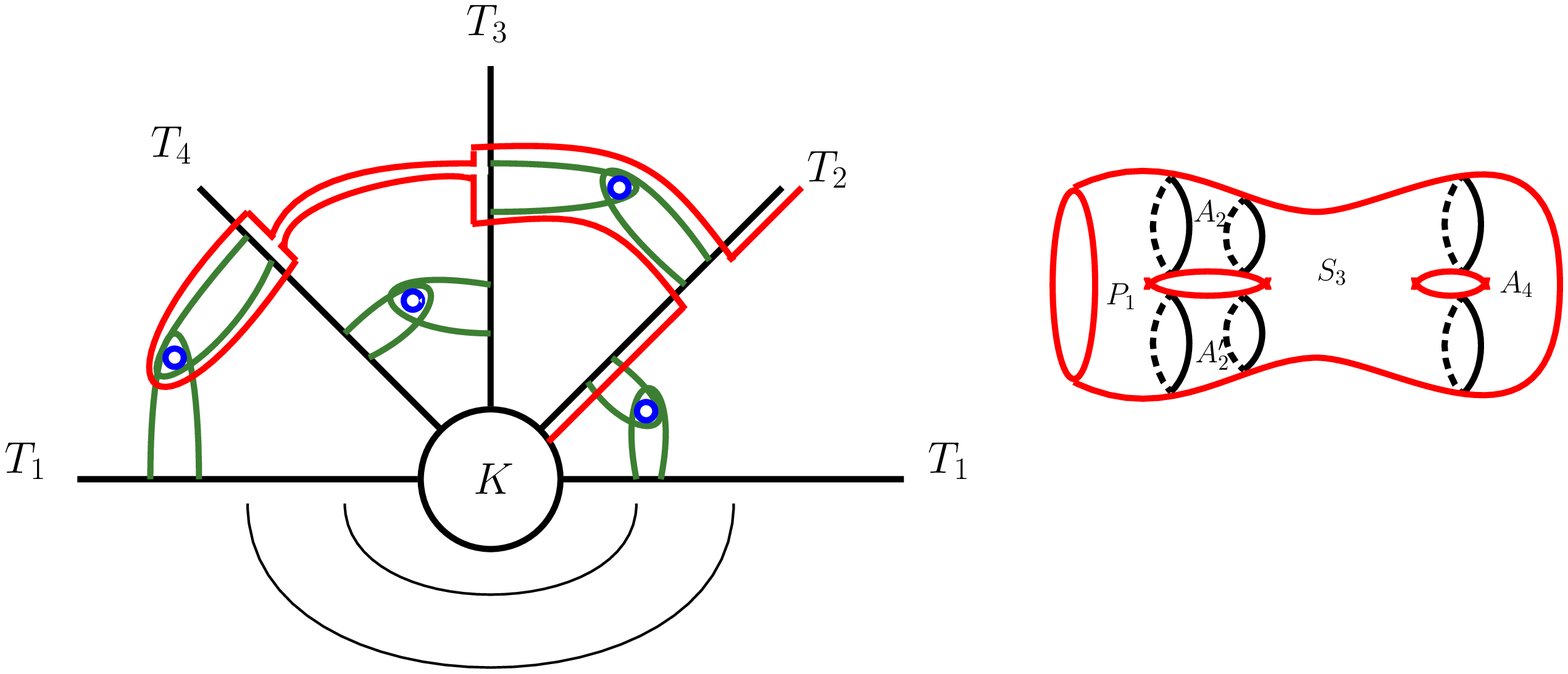}
\end{center}
\caption{}
\label{case3}
\end{figure}

This case is schematically shown in Figure \ref{case3}. In this case $\Sigma_2$ consists of the annuli $A_2$ and $A_2'$, which are spanning annuli in $H_2$, whose boundary curves are power circles in $H_2$. It follows that $P_1$ is a pair of pants whose boundaries are two curves in $T_2$, which are primitive circles in $H_1$ and one curve parallel to $K$, then it is possible to push $P_1$ into $H_2$ to eliminate $\Sigma\cap H_1$, a contradiction. %\textcolor{blue}{En este caso como que hay que cambiar la figura y poner a $P_1$ mas cerca de $T_2$, pa que se vea mas claro que son paralellos, o sea poner al circulito azul del otro lado de $P_1$.}

\vspace{1cm}

{\bf{Case 4: $A_4.P_3.A_2.S_1.A_2'$}}. 

\begin{figure}
\begin{center}
\includegraphics[scale=.46]{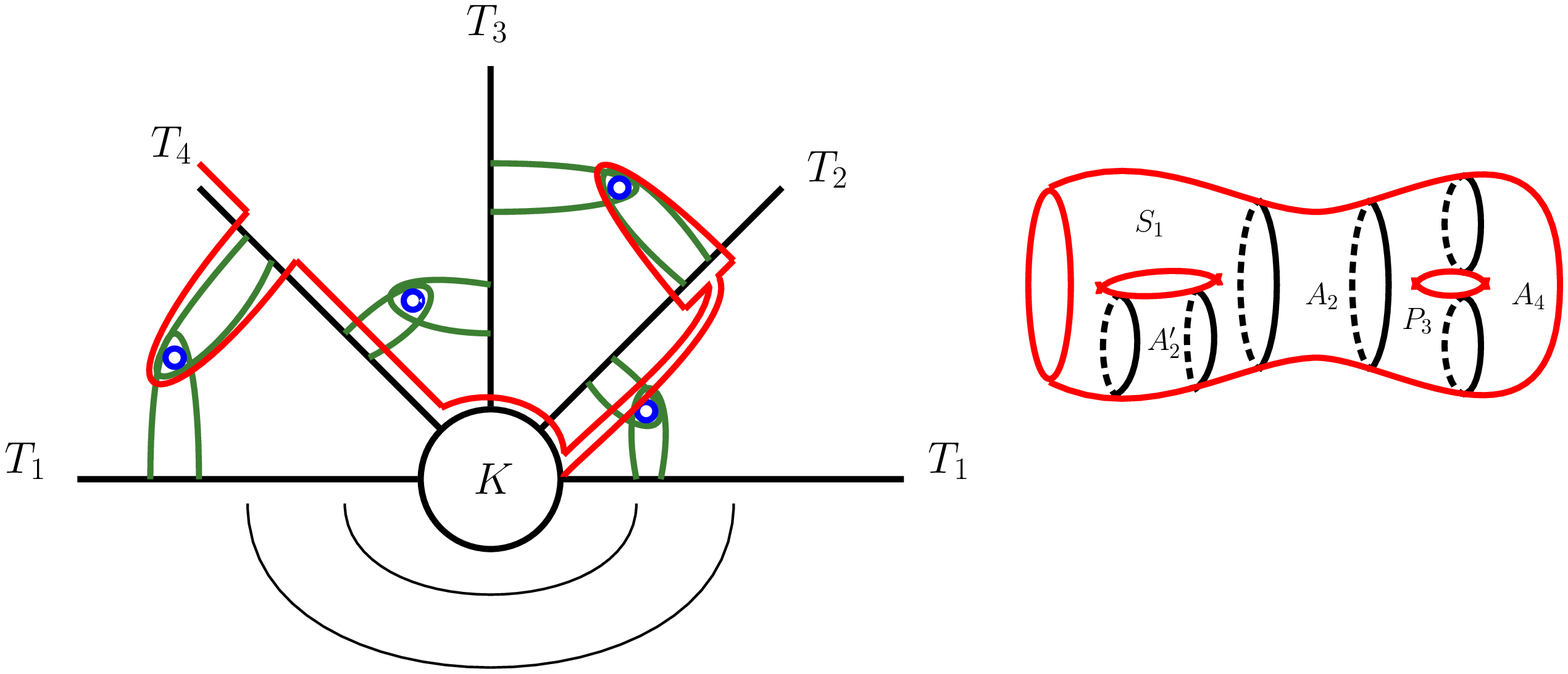}
\end{center}
\caption{Case 4}
\label{case4}
\end{figure}

%As $\Sigma$ is bicompressible, let $D$ be a compressing disk lying in the external side. Consider $D\cap\mathbb{T}$.
%Trivial closed intersection curves on $D$ can be eliminated by an isotopy. Then we can assume $D\cap\mathbb{T}$ consists of arcs.

%Take an innermost intersection arc $s$ on $D$, then there exists an arc $\sigma\subset\Sigma$ such that $\sigma\cup s$ cobound a disc $E\subset D$ and $E\cap \mathbb{T}=\emptyset$.
This case is schematically shown in Figure \ref{case4}. In this case $\Sigma_2$ two consists of two annuli, $A_2$ and $A_2'$, where $A_2$ is a boundary parallel annulus and $A_2'$ is parallel to the companion annulus in $H_2$ with boundary on $T_2$, and then $S_1$ is a 4-punctured sphere.

As in Case 1, let $D$ be a compressing disk for $\Sigma$ lying in the external side, such that $D\cap\mathbb{T}$ consists only of arcs. As before, take an innermost intersection arc $t$ on $D$, and an arc $\sigma\subset\Sigma$ such that $\sigma\cup t$ cobound a disk $E\subset D$ and ${int}E\cap \mathbb{T}=\emptyset$. 
Note that the cases where $t\subset T_4$ or $t\subset T_3$ and $\sigma\subset H_3$ can be eliminated as in case $1$. 

Now suppose that $t\subset T_3$ and $\sigma\subset H_2$. Remember that  $\Sigma_2$ is a disjoint union of $\Sigma_2=A_2\cup A_2'$, where $A_2$ is parallel to the annulus $J_2$ and $A_2'$ is parallel to the companion annulus in $H_2$, with its boundary lying in $T_2$. Then $\sigma\subset A_2$, but this is implies that $\sigma$ is not essential, a contradiction. 

Now suppose that $t\subset T_2$ and $\sigma \subset H_2$. Then $\sigma \subset A_2'$. Let $A$ be an annulus in $T_2$ with $\partial A=\partial A_2'$. Then $t\cup \sigma$ lies in $(T_2\backslash A)\cup A_2'$, which implies that $E$ is a compression disk for a surface parallel to $T_3$, which is not possible.

So we have that $t\subset T_2$ and $\sigma\subset H_1$. The four-punctured sphere $S_1$ has two boundary slopes $\partial_1 S_1$, $\partial_2 S_1$ which are primitive curves in $H_1$ and two boundary slopes parallel to $K$, one lying in $T_2$ ($\partial_3 S_1$) and the other in $\partial\mathcal{N}(K)$ ($\partial_4 S_1$). By making a $\partial$-compression in $S_1$ with the disk $E$ we obtain a pair of pants or an annulus and a pair of pants. We have the following cases:

\begin{itemize}
    \item If $t$ joins $\partial_1 S_1$ and $\partial_2 S_1$. By making a $\partial$-compression with $E$, we obtain a pair of pants whose boundaries are all parallel to $K$, which is not possible, by Lemma \ref{3pK}.
    
    \item If $t$ joins $\partial_1 S_1$ (or $\partial_2 S_1$) with $\partial_3 S_1$. By making a $\partial$-compression with $E$, we obtain a pair of pants with boundary components $\partial _2 S_1$, $\partial_4 S_1$ and one more curve parallel to $\partial_2 S_1$. As $\partial_2 S_1$ is a primitive curve in $H_1$, then the pair of pants is parallel into $T_2$.  It follows that it is possible to isotope $S_1$ to lie in $H_2$, a contradiction.
    
    \item If $t$ joins $\partial_1 S_1$ or $\partial_2 S_1$ with itself. By making a $\partial$-compression with $E$, we obtain an annulus and a pair of pants which are parallel into $T_2$ and which are possible to isotope to lie in $H_2$, a contradiction.
    
    \item If $t$ joins $\partial_3 S_1$ with itself. By making a $\partial$-compression with $E$, we obtain an annulus and a pair of pants which are parallel into $T_2$ and which are possible to isotope to lie in $H_2$, a contradiction.
\end{itemize}

\vspace{1cm}

{\bf{Case 5: $A_4.P_3.S_2.A_1$}}.

\begin{figure}
\begin{center}
\includegraphics[scale=.46]{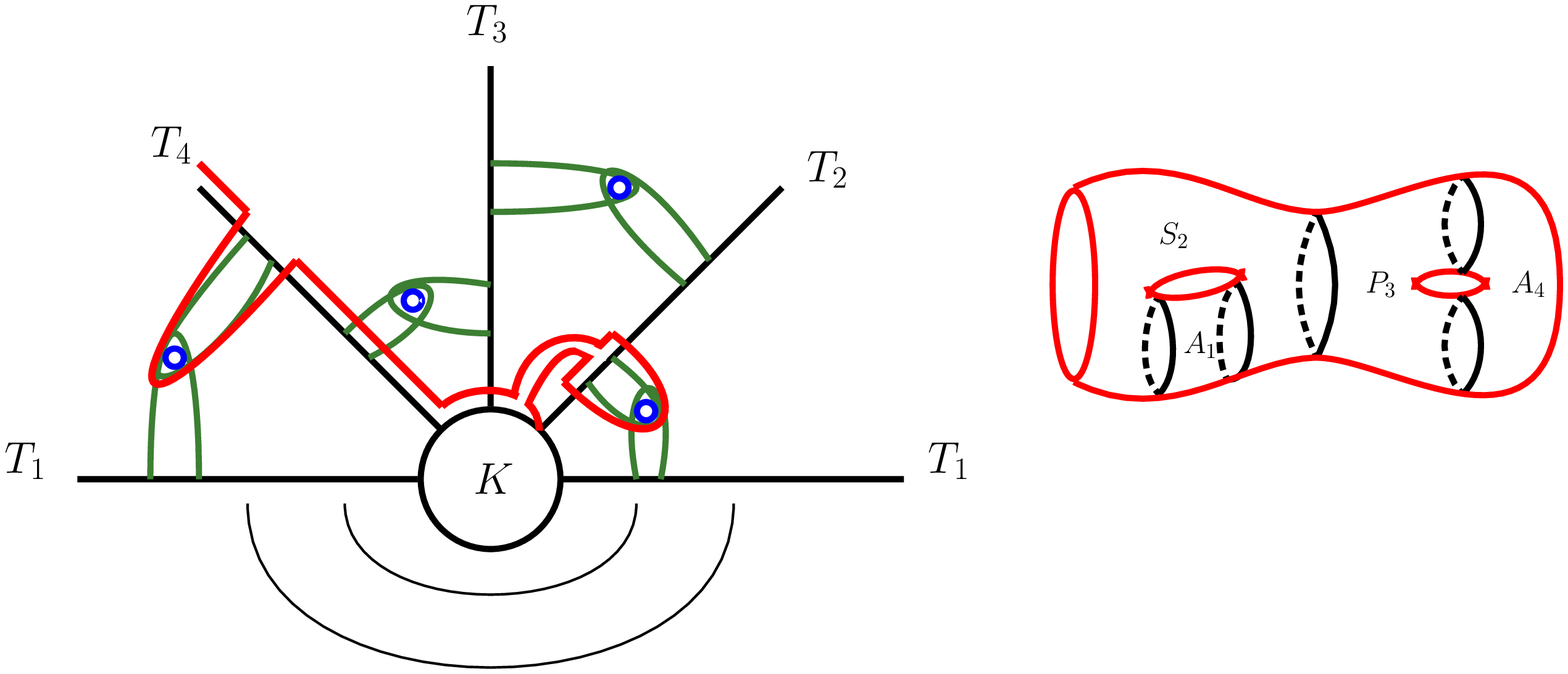}
\end{center}
\caption{}
\label{case5}
\end{figure}

%As $\Sigma$ is bicompressible, let $D$ be a compressing disk lying in the external side. Consider $D\cap\mathbb{T}$.
%Trivial closed intersection curves on $D$ can be eliminated by an isotopy. Then we can assume $D\cap\mathbb{T}$ consists of arcs.

%Take an innermost intersection arc $s$ on $D$, then there exists an arc $\sigma\subset\Sigma$ such that $\sigma\cup s$ cobound a disc $E\subset D$ and $E\cap \mathbb{T}=\emptyset$.
This case is schematically shown in Figure \ref{case5}. Led $D$, $E$, $t$ and $\sigma$ be as in Case 1.
Note that the cases where $t\subset T_4$ or $t\subset T_3$ and $\sigma\subset H_3$ can be eliminated as in Case $1$.

The four-punctured sphere $S_2$ has two boundary curves $\partial_1 S_2$, $\partial_2 S_2$, which lie in $T_2$ and are primitive curves in $H_2$ and two boundary curves parallel to $K$, one lying in $T_3$ ($\partial_3 S_2$) and the other in $\partial\mathcal{N}(K)$ ($\partial_4 S_2$).

Now we assume $t\subset T_3$ and $\sigma\subset H_2$. Then $t$ is an arc joining $\partial_3 S$ with itself. By making a $\partial$-compression in $S_2$ with $E$, we obtain an annulus and a pair of pants. The boundary of the annulus consists of an essential curve in $T_3$ and either $\partial_1 S_2$, $\partial_2 S_2$ or $\partial_4 S_2$. In any case, this is in contradiction with Lemma \ref{5anillos}.

Now suppose, $t\subset T_2$ and $\sigma\subset H_2$. If $t$ joins $\partial_1 S_2$ and $\partial_2 S_2$, then by doing a $\partial$-compression on $S_2$ with $E$, we obtain a pair of pants whose all boundaries are parallel to $K$, a contradiction to Lemma \ref{3pK}. Now, if $t$ joins $\partial_1 S_2$ (or $\partial_2 S_2$) to itself, by doing a $\partial$-compression to $S_2$ we obtain an annulus $A$ and a pair of pants $P$. The curve $\partial_1 S_2$ is split into two curves, say $\partial_{11}$ which is parallel to $\partial T_2$, and $\partial_{12}$ which is parallel to $\partial_2 S_2$. If the annulus $A$ has $\partial_{12}$ as a boundary component, then the other boundary component can not be $\partial_3 S_2$ or $\partial_4 S_2$, and then it has to be $\partial_2 S_2$. But in this case the three boundary components of $P$ are parallel to $K$, which is not possible. Then the annulus $A$ has $\partial_{11}$ as a boundary component. The other boundary component can not be $\partial_2 S_2$. It can not be $\partial_3 S_2$, for in this case $A$ and $P$ will intersect. Then the other boundary component of $A$ is $\partial_4 S_2$. In this case $A$ can be isotoped into $H_1$. So, after pushing $S_2$ into $H_1$ by using $E$, and then pushing $A$ into $H_1$, we get a configuration as in Case 1.

Finally, suppose that $t\subset T_2$ and $\sigma\subset H_1$. We note that $A_1$ is parallel to the companion annulus in $H_1$. Let $A$ be an annulus in $T_2$ with the same boundary as $A_1$. Then $E$ is a compression disk for the surface $(T_2-A)\cup A_1$, which is parallel to $T_1$, which is not possible.

\vspace{1cm}

{\bf{Case 6: $A_4.P_3.A_2.T_1'$}}.

\begin{figure}
\begin{center}
\includegraphics[scale=.46]{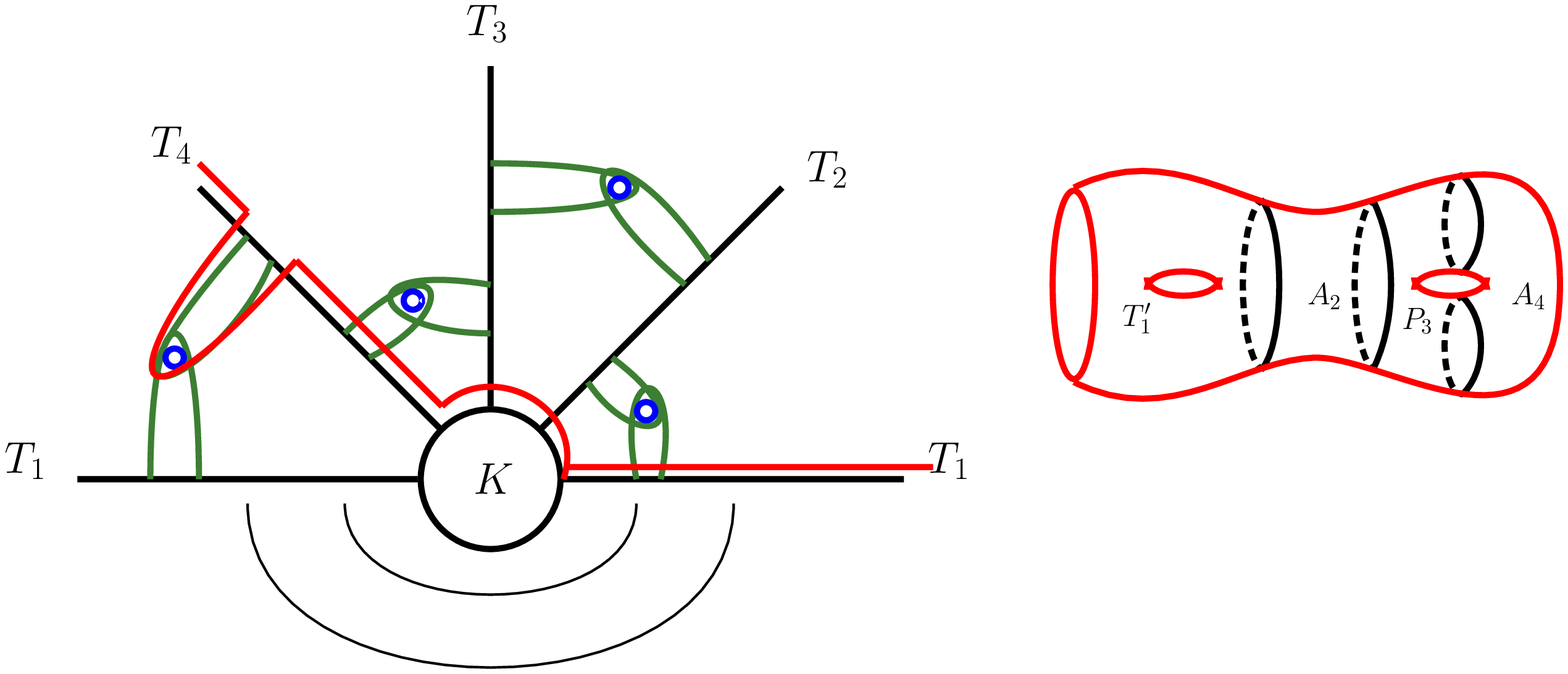}
\end{center}
\caption{}
\label{case6}
\end{figure}
This case is schematically shown in Figure \ref{case6}. Note that one boundary component of $T_1'$ lies in $J_1$, which is precisely $\partial \Sigma$, and the other lies in $T_2$ and it is parallel to $\partial T_2$. Led $D$, $E$, $t$ and $\sigma$ be as in Case 1.
Note that the cases where $t\subset T_4$ or $t\subset T_3$ can be eliminated as in Case $1$. The only remaining case is when $t \subset T_2$ and $\sigma \subset H_1$. In this case the arc $t$ joins points in the same boundary component of $T_1'$. By pushing $T_1'$ along the disk $E$, we get a new configuration for $\Sigma$, $\Sigma \cap H_3$ and $\Sigma \cap H_4$ are both a pair of pants, that is, we have  a configuration as in Case $1$, which is already solved. 

\end{proof}

%\begin{teo}
%There exists an infinite family of knots $K(p)$, $p\in\mathbb{Z}$, whose exterior contains $4$ mutually disjoint and non-parallel genus one Seifert surfaces $T_{i}$, $i=1,\ldots, 4$; $T_{i+1}$ obtained from $T_{i}$ by changing a power annulus in $T_{i}$ for the companion annulus and $E(K)$ admits a circular decomposition of width $\{3,3\}$.
%\end{teo}

\begin{teo}\label{main}
There exists an infinite family of genus one hyperbolic knots which are not almost-fibered. 
\end{teo}

\begin{proof}
Let $K(\ell,m,n)$ be the family of knots constructed in Theorem \ref{familyknots}. These knots have genus one and by Theorem \ref{familywidth} their exterior admits a circular decomposition of width $\{3,3\}$. Also, by Lemma \ref{hyperbolic} the knots $K(\ell,m,n)$ are hyperbolic. If they are almost-fibered, then their exterior would admit a circular decomposition of width $\{3\}$ which is impossible by Theorem \ref{principal}.
\end{proof}

\vskip20pt
\textbf{Acknowledgements.}
The first author is partially supported by grant PAPIIT-UNAM IN116720.
The second author's research is supported by a CONACYT postdoctoral fellowship. 
The third author made this research on a sabbatical year at IMATE-UNAM, Unidad Juriquilla.

\end{document}